\documentclass[12pt,a4paper]{amsart}

\usepackage{mystyle,mymacros}
\usepackage[all]{xy} 		
\usepackage{tabularx}
\usepackage{enumerate}
\usepackage{times}
\usepackage[ruled,vlined]{algorithm2e}
\usepackage{float}
\usepackage{tikz,graphicx}
\graphicspath{ {./images/} }
\usepackage{hyperref}

\title{Explicit Vologodsky Integration for Hyperelliptic Curves}

\author{Enis Kaya}
\address{Enis Kaya, Max Planck Institute for Mathematics in the Sciences, Inselstrasse 22, 04103, Leipzig, Germany}
\email{enis.kaya@mis.mpg.de}

\newcommand{\V}{{\operatorname{Vol}}}
\newcommand{\BC}{{\operatorname{BC}}}

\newcommand{\lV}{{\operatorname{Vol\ }}}
\newcommand{\lBC}{{\operatorname{BC\ }}}
\newcommand{\dR}{{\operatorname{dR}}}
\def\presuper#1#2%
  {\mathop{}%
   \mathopen{\vphantom{#2}}^{#1}%
   \kern-12\scriptspace%
   #2}
\newcommand{\defi}[1]{\textbf{#1}} 

\newcommand\BCint{\presuper\BC\int}
\newcommand\lBCint{\presuper\lBC\int}

\newcommand\Vint{\presuper\V\int}
\newcommand\lVint{\presuper\lV\int}
\newcommand\tint{\presuper{t}\int}
\newcommand\ltint{\presuper{t\ }\int}

\DeclareMathOperator{\Log}{Log}

\mathtext{Diff}

\begin{document}

\begin{abstract} Vologodsky's theory of $p$-adic integration plays a central role in computing several interesting invariants in arithmetic geometry. In contrast with the theory developed by Coleman, it has the advantage of being insensitive to the reduction type at $p$. Building on recent work of Besser and Zerbes, we describe an algorithm for computing Vologodsky integrals on bad reduction hyperelliptic curves. This extends previous joint work with Katz to all meromorphic differential forms. We illustrate our algorithm with numerical examples computed in Sage.
\end{abstract}

\maketitle
\setcounter{tocdepth}{1}
\tableofcontents

\section{Introduction}

Coleman integration \cite{coleman1982dilogarithms,coleman_shalit88:p_adic_regulator,Bes02} is a method for defining $p$-adic iterated integrals on rigid analytic spaces associated to varieties with good reduction at $p$. Vologodsky integration \cite{vologodsky2003hodge} also produces such integrals on varieties, but it does not require that the varieties under consideration have good reduction at $p$. These two integration theories are both path-independent and they are known to be the \emph{same} in the case of good reduction. Therefore, Vologodsky integration is, in a sense, a generalization of Coleman integration to the bad reduction case.

On the other hand, by a cutting-and-pasting procedure, one can define in a natural way integrals on curves\footnote{In what follows we will only consider curves, not higher-dimensional varieties.} of bad reduction. More precisely, one can cover such a curve by \emph{basic wide open spaces}, certain rigid analytic spaces, each of which can be embedded into a curve of good reduction; and by performing the Coleman integrals there, one can piece together an integral. This naive integral, which will be referred to as the Berkovich--Coleman integral (see Section~\ref{Berk-Col-int-section}), is generally path-dependent and hence disagrees with the Vologodsky integral. 

As described below, Coleman and Vologodsky integration have numerous applications in arithmetic geometry, some of which rely on explicitly computing integrals. There are practical algorithms in the good reduction case when $p$ is odd: see \cite{BBKExplicit,BBComputing,balakrishnan2015coleman,best2019explicit} (single integrals on hyperelliptic curves), \cite{balakrishnan2013iterated,balakrishnan2015coleman} (iterated integrals on hyperelliptic curves), and \cite{balakrishnan2020explicit} (single integrals on smooth curves). All of this relies on an algorithm for explicitly computing the action of Frobenius on $p$-adic cohomology in order to realize the principle of \emph{analytic continuation along Frobenius}: see \cite{kedlaya2001counting,harvey2007kedlayas,harrison2012extension} (hyperelliptic curves), and \cite{Tuitman:counting1,Tuitman:counting2} (smooth curves). 

On the other hand, Vologodsky integration has been, so far, difficult to compute. The main obstacle is that Vologodsky's original construction is not quite suitable for machine computation. Nevertheless, in joint work with Katz \cite{katz2020padic}, we proved that ($p$-adic) abelian integration\footnote{Thanks to Zarhin and Colmez \cite{zarhin96:abelian_integral,colmez1998integration}, single Vologodsky integrals were known before Vologodsky's work; the case of holomorphic forms is referred to as abelian integration.} on curves with bad reduction \emph{reduces} to Berkovich--Coleman integration (see \cite[Theorem~3.16]{katz2020padic} for a precise formulation); and this result, when combined with the Coleman integration algorithms, allowed us to give an algorithm for computing abelian integrals on bad reduction hyperelliptic curves when $p\neq 2$ (see \cite[Algorithm~7]{katz2020padic}). Based on this algorithm, we provided several numerical examples carried out using Sage (see \cite[Section~9]{katz2020padic}).

The paper \cite{katz2020padic}, to the best of our knowledge, is the first attempt in the literature to compute $p$-adic integrals in the bad reduction case. Yet, it has the drawback that it only deals with holomorphic forms; however, for some applications such as $p$-adic heights on curves, it is necessary to work with more general meromorphic forms. In the present paper, we extend the techniques of \emph{loc. cit.} in a natural way to also cover meromorphic forms, building on the recent work of Besser and Zerbes \cite{BesserZerbes} which \emph{relates} single Vologodsky integration on bad reduction curves to Coleman primitives. In particular, we present an algorithm for computing single Vologodsky integrals on bad reduction hyperelliptic curves for $p\neq 2$ (Algorithm~\ref{integrationalgorithm}). For a fixed hyperelliptic curve $X$ with affine model $y^2 = f(x)$, the algorithm roughly proceeds in the following steps:
\begin{enumerate}
    \item Based on a comparison formula, which follows from the work of Besser and Zerbes \cite{BesserZerbes}, we reduce the problem to the computation of certain Berkovich--Coleman integrals on $X$ (Theorem~\ref{mainthm}).
    \item By viewing $X$ as a double cover of the projective line $\P^1$ and examining the \emph{roots} of $f(x)$, we construct a covering $\cD$ of the curve $X$ by basic wide open spaces (Section~\ref{Coverings}).\footnote{This is essentially equivalent to constructing a semistable model of $X$.}
    \item Using additivity of Berkovich--Coleman integrals under concatenation of paths, we reduce the computation of the integrals in Step~(1) to the computation of various Coleman integrals on certain elements of $\cD$. In order to compute these integrals, we first embed the elements of $\cD$ of interest into good reduction hyperelliptic curves. Then, by a pole reduction argument, we rewrite the corresponding differentials in simplified forms. Finally, we make use of the known Coleman integration algorithms (Section~\ref{BC-int-hypcur}).
\end{enumerate}
Based on this algorithm, we give two numerical examples computed in Sage (Section~\ref{NumericalExamples}).

We note that Sage includes implementations of the Coleman integration algorithms in \cite{BBKExplicit, BBComputing,balakrishnan2015coleman}. In fact, this is the main reason why we use Sage. Implementing our algorithm in Sage, however, seems out of reach at present, even if we take the base field to be $\Q_p$. Here are two main reasons:
\begin{itemize}
    \item The implementations mentioned above assume that the curve in question is defined over $\Q_p$, but the curves we get in Step~(3) above might be defined over non-trivial finite extensions of $\Q_p$.
    \item In Sage, one can define and work with Eisenstein and unramified extensions separately; but, neither conversion between these extensions nor general extensions are available. In our approach, however, it may be necessary to work with several finite extensions of $\Q_p$ at the same time.
\end{itemize}

In a sense, the present paper is a sequel to \cite{katz2020padic}. Indeed, our new comparison formula and algorithm are slight modifications of the corresponding ones in \cite{katz2020padic}; therefore, we frequently refer to \cite{katz2020padic}. On the other hand, the present paper provides an additional result (Section~\ref{Error Bounds}); this result bounds the error term in the method for computing Berkovich--Coleman integrals, which involves an expansion in an infinite series. The main contributions are in Sections~\ref{ComparisonofTheIntegrals}, \ref{BC-int-hypcur}, and \ref{TheAlgorithm}.

\subsection{Applications} It is worth highlighting a few of the numerous applications of explicit Vologodsky integration. Let $X$ be a smooth, proper, and geometrically connected curve over the field $\Q$ of rational numbers of genus $g\geq 1$. For brevity, the base field is $\Q$; but everything we will say admits a generalization to a number field. We asssume, for simplicity, that $X$ has a $\Q$-rational point. Let $p$ be an odd prime.

\subsubsection*{$p$-adic height pairings} When the curve $X$ has good reduction at $p$, Coleman and Gross \cite{coleman1989} gave a construction of a $p$-adic height pairing on $X$ which is, by definition, a sum of local height pairings at each prime number. The local components away from $p$ are described using arithmetic intersection theory, and the local component at $p$ is given in terms of the Coleman integral of a non-holomorphic differential. Besser \cite[Theorem~1.1]{Bes04} showed that the Coleman--Gross pairing is \emph{equivalent} to the $p$-adic height pairing of Nekov\'a\v r \cite{Nek93}.

In another direction, replacing Coleman integration in the recipe above by Vologodsky integration, Besser \cite[Definition~2.1]{BesserPairing} gave an extended definition of the Coleman--Gross pairing on $X$ without any assumptions on the reduction at $p$. He also proved that this new pairing is still \emph{equivalent} to the one defined by Nekov\'a\v r when the reduction is semistable; see \cite[Theorem~1.1]{BesserPairing}.\footnote{In the literature, there are several definitions of $p$-adic height pairings; we restrict attention to those that are relevant to our setting.}

In the case that $X$ is a hyperelliptic curve with good reduction at $p$, an algorithm to compute the local height pairing at $p$ was provided in Balakrishnan--Besser \cite{BBComputing}. The techniques of the current article make it possible to remove the good reduction assumption from this setting and we are currently working on this with M{\"u}ller, but see Example~\ref{ex1} for an illustration in the elliptic curve case.

\subsubsection*{$p$-adic regulators} A $p$-adic analogue of the conjecture of Birch and Swinnerton--Dyer (BSD) for an elliptic curve over $\Q$ was given in Mazur--Tate--Teitelbaum (MTT) \cite{mazur1986p} when $p$ is a prime of good ordinary or multiplicative reduction, with the canonical regulator replaced by a $p$-adic regulator. Balakrishnan, M{\"u}ller, and Stein \cite{balakrishnan2016p} formulated a generalization of the MTT conjecture in the good ordinary case to higher dimensional (modular) abelian varieties over $\Q$, where the $p$-adic regulator is given in terms of Nekov\'a\v r's height pairing \cite{Nek93}. They also provided numerical evidence supporting their conjecture for Jacobians of hyperelliptic curves. This involves computing Coleman--Gross pairings on hyperelliptic curves, and, as explained above, the $p$-part of such a computation is based on \cite{BBComputing}.

The MTT conjecture in the case of split multiplicative reduction, the \emph{exceptional} case, is of special interest. One might expect that a generalization of this conjecture to higher dimensional abelian varieties over $\Q$ in the case of split purely toric reduction can be formulated. In the spirit of the preceding paragraph, formulating such a conjecture, as well as gathering numerical evidence for it, require the computation of extended Coleman--Gross pairings on hyperelliptic curves. The forthcoming work with M{\"u}ller mentioned above will handle the $p$-part of such a computation.

\subsubsection*{(Non-abelian) Chabauty} Hereafter, we consider the case where $g\geq 2$. By Faltings' theorem, the set $X(\Q)$ of rational points on $X$ is known to be finite; however, his proof is ineffective and at present no general algorithm for the computation of $X(\Q)$ is known. 

The method of Chabauty--Coleman \cite{mccallum2012method} is a $p$-adic method that attempts to \emph{determine} $X(\Q)$ under the condition that $r<g$, where $r$ is the Mordell--Weil rank of the Jacobian of $X$. This method relies heavily on computing abelian integrals on $X$. In the special case that $X$ is a hyperelliptic\footnote{In fact, the recent work of Balakrishnan--Tuitman \cite{balakrishnan2020explicit} overcame the hyperelliptic restriction.} curve having good reduction at $p$, this can be achieved using the algorithms developed by Balakrishnan--Bradshaw--Kedlaya and Balakrishnan in \cite{BBKExplicit,balakrishnan2015coleman}. For the case in which $X$ does not necessarily have good reduction at $p$ (but is still hyperelliptic), the techniques developed in \cite{katz2020padic} can be employed; see Example~9.4 of \emph{loc. cit.} for an illustration.

The hypothesis that $r<g$ in the Chabauty--Coleman method plays an essential role. An approach to circumvent this limitation is Kim's non-abelian Chabauty \cite{kim05:motivic_fundamental_group,kim09:unipotent_albanese,kim2010massey}, of which quadratic Chabauty is a special case. There is great interest in making the quadratic Chabauty method explicit; see, for example, \cite{balakrishnan2016quadratic,balakrishnan2017computing,balakrishnan2018,balakrishnandogra2021quadratic,balakrishnansplitcartan,bianchi2020quadratic,balakrishnan2021explicit,balakrishnan2021quadratic}. 

One of the crucial steps in carrying out non-abelian Chabauty is to compute iterated integrals. Balakrishnan \cite{balakrishnan2013iterated,balakrishnan2015coleman} gave algorithms to compute double Coleman integrals on good reduction hyperelliptic curves. Combining our techniques with the forthcoming work of Katz and Litt \cite{katzlitt}, one should be able to compute iterated Vologodsky integrals on semistable hyperelliptic curves. The quadratic Chabauty method also requires the computation of $p$-adic heights, which we have already discussed.

\subsection{Why Bad Primes?} Why, one might ask, are primes of bad reduction interesting or useful? It would be hopeless to give a complete answer to this philosophical question here; however, we will make a few comments.

\subsubsection*{Local problems} Techniques developed in order to deal with $p$-adic problems (i.e., problems whose objects are defined over $p$-adic fields) depend, in general, on the \emph{nature} of the prime $p$. One striking example of this phenomenon is the $p$-adic BSD conjecture for elliptic curves. For different reduction types, this conjecture is (in general) of a quite different nature and of interest in its own right; see Stein--Wuthrich \cite{stein2013algorithms}.

\subsubsection*{Global problems} We sometimes study a problem over $\Q$ by studying it over $\Q_p$ for a fixed prime $p$. A \emph{good} prime is often more convenient to work with; however, for practical purposes, we may need to take $p$ as small as possible and this additional condition might force us to work with a \emph{bad} prime. As an illustration, we may look at \cite[Example~5.1]{katz_dzb13:chabauty_coleman}, in which Katz--Zureick-Brown showed, for a certain curve $X/\Q$ with bad reduction at $5$, that $5$ is the only prime that allows to compute $X(\Q)$ using Coleman's upper bound for $\#X(\Q)$ \cite{coleman85:effective_chabauty} and/or its refinements \cite{lorenzin_tucker02:chabauty_coleman,stoll_2006,katz_dzb13:chabauty_coleman}.

A related topic is the uniformity conjecture \cite{caporaso1997uniformity}. This conjecture is one of the outstanding conjectures in Diophantine geometry, and asserts that there exists a constant $B(\Q,g)$ such that every smooth curve $X/\Q$ of genus $g\geq 2$ has at most $B(\Q,g)$ rational points. The first result along these lines is due to Stoll \cite{stoll2018uniform}, who proved uniform bounds for hyperelliptic curves of small rank; his result was later generalized by Katz, Rabinoff, and Zureick-Brown \cite{krzb15:uniform_bounds} to arbitrary curves of small rank. See also the paper of Kantor \cite{kantor2017rankfavorable}. Here, to get a bound \emph{independent of the geometry of curves}, it is necessary to work with primes of bad reduction.

\subsection{Outline} First, in Section 2, we introduce some notation and recall several basic facts that we will need. In Section 3, after recalling Berkovich--Coleman and Vologodsky integration, we prove a result comparing the two. Section 4 summarizes how to construct coverings of hyperelliptic curves. In Section 5, we explain how to compute Berkovich--Coleman integrals on hyperelliptic curves, and we describe in Section 6 the algorithm for computing Vologodsky integrals on such curves. Finally, in Section 7 we conclude with numerical examples illustrating our methods.

\subsection*{Acknowledgements} We would very much like to thank Steffen M{\"u}ller for a continuous stream of discussions regarding this work. Besides, we thank Jennifer Balakrishnan, Amnon Besser, Francesca Bianchi, Ayhan Dil, Netan Dogra, Stevan Gajovi\'c, Eric Katz and Ken McMurdy for helpful conversations, either directly or by email. We also would like to thank the anonymous referee for several useful suggestions. The author was partially supported by NWO grant 613.009.124.

\section{Preliminaries}

Throughout we work with a fixed odd prime number $p$. Let $\C_p$ denote the completion of an algebraic closure of the field of $p$-adic numbers $\Q_p$. Let $v_p$ be the valuation on $\C_p$ normalized so that $v_p(p)=1$; it corresponds to the absolute value $\|\cdot\|_p$ where $\|\cdot\|_p=p^{-v_p(\cdot)}$. Let $\bK$ denote a finite extension of $\Q_p$ with ring of integers $\mathcal{O}_{\bK}$ and residue field $\bk$.

\subsection{Rigid Analysis} We will make use of several concepts from rigid analytic geometry. Some standard references are the book of Bosch--G\"untzer--Remmert \cite{bosch_guntzer_remmert84:non_archimed_analysis} and the book of Fresnel--van der Put \cite{fresnel_put04:rigid_analytic_geometry} (see also Schneider \cite{schneider_1998} for a quick introduction).

We will use the marker ``an'' to denote rigid analytification; in particular, $\A^{1,\an}$ denotes the rigid affine line. Set
\begin{eqnarray*}
B(a,r)&=&\{z\in\A^{1,\an} \mid \|z-a\|_p<r\},\ \ \ r>0;\\
\overline{B}(a,r)&=&\{z\in\A^{1,\an} \mid \|z-a\|_p\leq r\},\ \ \ r>0.
\end{eqnarray*} 
An open (resp. closed) disc is a rigid space isomorphic to some $B(a,r)$ (resp. $\overline{B}(a,r)$). An (open) annulus is one isomorphic to
\[\{z\in\A^{1,\an} \mid s<\|z-a\|_p<r\},\ \ \ 0<s<r.\]

A \defi{wide open space} is a rigid analytic space isomorphic to the complement in a smooth complete connected curve of finitely many closed discs. Examples of wide opens include open discs and annuli. A wide open space is \defi{basic} if it is isomorphic to the complement in a good reduction complete connected curve of finitely many closed discs, each of which is contained in a distinct residue disc. These spaces are the building blocks of Coleman's theory. In this paper, two types of basic wide opens will be particularly important. We call a basic wide open space \defi{rational} (resp. \defi{hyperelliptic}) if it lies in the rigid analytification of $\P^1$ (resp. a hyperelliptic curve).

\subsection{Differential Forms} Let $X/\bK$ be a smooth, proper, geometrically connected curve of genus $g$. A meromorphic $1$-form on $X$ over $\bK$ is said to be of the \defi{first kind} if it is holomorphic, and of the \defi{second kind} if it has residue $0$ at every point. For instance, exact differentials, i.e., differentials of rational functions, are of the second kind. The differentials of the second kind modulo exact differentials form a $2g$-dimensional $\bK$-vector space. It is canonically isomorphic to the first algebraic de Rham cohomology $H_{\dR}^1(X/\bK)$ of $X/\bK$, which is the first hypercohomology group of the de Rham complex
\[0\rightarrow \mathcal{O}_X \rightarrow \Omega_{X/\bK}^1\rightarrow 0\]
on $X/\bK$. We have a canonical exact sequence
\[0\rightarrow H^0(X,\Omega_{X/\bK}^1) \rightarrow H_{\dR}^1(X/\bK) \rightarrow H^1(X,\mathcal{O}_X) \rightarrow 0.\]
We identify the space $H^0(X,\Omega_{X/\bK}^1)$ of differentials of the first kind with its image; it is $g$-dimensional and we denote it by $H_{\dR}^{1,0}(X/\bK)$.

We say a meromorphic $1$-form on $X$ over $\bK$ is of the \defi{third kind} if it is regular, except possibly for simple poles with integer residues. The logarithmic differentials, i.e., those of the form $df/f$ for $f\in \bK(X)^{\times}$, are of the third kind. Let $T(\bK)$ denote the subgroup of differentials of the third kind and $\operatorname{Div}_{\bK}^0(X)$ the group of divisors of degree zero on $X$ over $\bK$. The residual divisor homomorphism
\[\Res\colon T(\bK)\to \operatorname{Div}_{\bK}^0(X),\ \ \ \nu\mapsto \sum_P \Res_P \nu \cdot (P)\]
where the sum is taken over closed points of $X$, gives rise to the following exact sequence:
\begin{equation} \label{resdivhom}
    0\rightarrow H_{\dR}^{1,0}(X/\bK) \rightarrow T(\bK) \rightarrow \operatorname{Div}_{\bK}^0(X) \rightarrow 0.
\end{equation}

\section{$p$-adic Integration Theories}
\label{integrationsection}

In this section, we define Vologodsky and Berkovich--Coleman integration. We also give a formula for passing between them when the underlying space is a curve; this generalizes \cite[Theorem~3.16]{katz2020padic} in which only holomorphic $1$-forms are considered.

Choose once and for all a branch of the $p$-adic logarithm, i.e., a homomorphism
\[ \Log \colon \C_p^\times \To \C_p \] given by the Mercator series $\text{Log}(1+z)=z-\frac{z^2}{2}+\frac{z^3}{3}-\cdots$ when $\|z\|<1$. A branch is determined by specifying $\text{Log}(p)$ in $\C_p$. 

\subsection{Vologodsky Integration} For a smooth, geometrically connected algebraic $\bK$-variety $X$, we let $Z^1_{\dR}(X)$ denote the space of closed $1$-forms on $X$.

\begin{thm} There is a unique way to construct, for every smooth, geometrically connected algebraic $\bK$-variety $X$, every $\omega\in Z^1_\dR(X)$ and every pair of points $x,y\in X(\bK)$, an integral
\[\int_x^y \omega \in \bK\]
such that the following are true:
\begin{enumerate}
    \item If $x_0\in X(\bK)$, then
    \[F\colon X(\bK)\to \bK, \ \ \ x\mapsto \int_{x_0}^x \omega\]
    is a locally analytic function satisfying $dF = \omega$.
    \item If $\omega_1,\omega_2\in Z^1_\dR(X)$ and $c_1,c_2\in \bK$, then
    \[\int_x^y (c_1\omega_1+c_2\omega_2) = c_1\int_x^y \omega_1 + c_2\int_x^y \omega_2.\]
    \item If $x,y,z\in X(\bK)$, then 
    \[\int_x^z \omega = \int_x^y \omega + \int_y^z \omega.\]
    \item If $f$ is a rational function on $X$, then
    \[\int_x^y df = f(y)-f(x)\]
    provided that $f$ is defined on the endpoints.
    \item If $f$ is a non-zero rational function on $X$, then
    \[\int_x^y \frac{df}{f} = \Log\left(\frac{f(y)}{f(x)}\right)\]
    provided that $f$ is defined and non-zero on the endpoints.
    \item If $h\colon X\to Y$ is a $\bK$-morphism, $x,y\in X(\bK)$ and $\omega\in Z^1_\dR(Y)$, then
    \[\int_x^y h^*\omega = \int_{h(x)}^{h(y)} \omega.\]
\end{enumerate}
\end{thm}

This theorem was proved by Colmez \cite[Th\'{e}or\`{e}me~1]{colmez1998integration} and a generalization was given later by Vologodsky \cite[Theorem~B]{vologodsky2003hodge}. Following Besser and Zerbes \cite{BesserZerbes,BesserPairing}, we call this integral the \defi{Vologodsky integral} and denote it by $\lVint$.

\subsection{Berkovich--Coleman Integration} \label{Berk-Col-int-section} In order to study Berkovich--Coleman integration, we will follow Berkovich's language; in principle this can be avoided by using intersection theory on semistable curves, however, the analytic framework seems much more natural. The canonical references are Berkovich's book \cite{berkovic90:analytic_geometry} and the paper of Baker--Payne--Rabinoff \cite{baker_payne_rabinoff13:analytic_curves} (see also \cite[Section~3]{krzb:utahsurvey} for a summary of the basics that will be used below).

For a smooth $\C_p$-analytic space $X$, we let $Z^1_{\dR}(X)$ denote the space of closed analytic $1$-forms on $X$, and we let $\cP(X)$ denote the set of paths $\gamma\colon [0,1]\to X$ with ends in $X(\C_p)$. The following is a special case of a theorem of Berkovich \cite[Theorem~9.1.1]{berkovic07:integration}. 

\begin{thm} There is a unique way to construct, for every smooth $\C_p$-analytic space $X$, every $\omega\in Z^1_\dR(X)$ and every $\gamma\in \cP(X)$, an integral
\[\int_\gamma \omega \in \C_p\]
such that the following are true:
\begin{enumerate}
    \item If $\omega_1,\omega_2\in Z^1_\dR(X)$ and $c_1,c_2\in \C_p$, then
    \[\int_\gamma (c_1\omega_1+c_2\omega_2) = c_1\int_\gamma \omega_1 + c_2\int_\gamma \omega_2.\]
    \item If $\gamma_1,\gamma_2\in\cP(X)$ with $\gamma_1(1)=\gamma_2(0)$, then \[\int_{\gamma_1\gamma_2}\omega = \int_{\gamma_1} \omega + \int_{\gamma_2}\omega\] where $\gamma_1\gamma_2$ is the concatenation.
    \item If $\gamma_1,\gamma_2\in\cP(X)$ are homotopic with fixed endpoints, then
    \[\int_{\gamma_1} \omega = \int_{\gamma_2} \omega.\]
    \item If $f$ is a rational function on $X$, then
    \[\int_\gamma df = f(\gamma(1))-f(\gamma(0))\]
    provided that $f$ is defined on the ends of $\gamma$.
    \item If $f$ is a non-zero rational function on $X$, then
    \[\int_\gamma \frac{df}{f} = \Log\left(\frac{f(\gamma(1))}{f(\gamma(0))}\right)\]
    provided that $f$ is defined and non-zero on the ends of $\gamma$.
    \item If $h\colon X\to Y$ is a morphism, $\gamma\in \cP(X)$ and $\omega\in Z^1_\dR(Y)$, then
    \[\int_\gamma h^*\omega = \int_{h\circ\gamma} \omega.\]
\end{enumerate}
\end{thm}

This integration was first developed for curves by Coleman \cite{coleman1982dilogarithms,coleman85:torsion} and Coleman--de Shalit \cite{coleman_shalit88:p_adic_regulator}. Following Katz, Rabinoff, and Zureick-Brown \cite{krzb15:uniform_bounds,krzb:utahsurvey},  we call this integral the \defi{Berkovich--Coleman integral} and denote it by $\lBCint$. We call a Berkovich--Coleman integral along a closed path a \defi{Berkovich--Coleman period}.

When $X$ is simply-connected, in which case the Berkovich--Coleman integral is path-independent, we may simply write $\lBCint_x^y = \lBCint_\gamma$ for any path $\gamma$ from $x$ to $y$. 

For contrast with the Vologodsky integral, the Berkovich--Coleman integral is \emph{local}, i.e., if $U\subset X$ is an open subdomain and $\gamma([0,1])\subset U$, then the integral $\lBCint_\gamma\omega$ can be computed from $U$, $\omega|_U$ and $\gamma$.

\subsection{Comparison of the Integrals} \label{ComparisonofTheIntegrals} In this subsection, we fix a smooth, proper and geometrically connected (algebraic) curve $X$ over $\bK$. On this space, we have the Vologodsky integration $\lVint$; on the other hand, we have the Berkovich--Coleman integration $\lBCint$ on its analytification, in the sense of Berkovich \cite{berkovic90:analytic_geometry}. The relation between the two will be provided by a combinatorial object which is quite easy to compute in practice, namely \emph{tropical integration}.

\subsubsection*{Tropical Integration} We begin by setting some notation and conventions. Let $\Gamma$ be a finite and connected graph. We will denote by $V(\Gamma)$ and $E(\Gamma)$, respectively, the set of vertices and oriented edges. For $e\in E(\Gamma)$, we write $i(e)$ and $t(e)$ for the initial and terminal point of $e$, respectively; and we denote by $-e$ the same edge as $e$ with the reverse orientation. 

The $\C_p$-vector space generated by $V(\Gamma)$ is called $0$-chains with coefficients in $\C_p$ and denoted by $C_0(\Gamma;\C_p)$; and the $\C_p$-vector space generated by $E(\Gamma)$ subject to the relations
\[e + (-e) = 0 \ \text{for each edge } e\]
is called $1$-chains with coefficients in $\C_p$ and denoted by $C_1(\Gamma;\C_p)$. As $C_0(\Gamma;\C_p)$ is canonically isomorphic to its dual $C_0(\Gamma;\C_p)^*$, we may identify a $0$-chain $C=\sum_{v} c_vv$ with the function $C\colon V(\Gamma)\to \C_p$ given by $v\mapsto c_v$. A similar remark applies to $C_1(\Gamma;\C_p)$. The boundary map $d\colon C_1(\Gamma;\C_p)\to C_0(\Gamma;\C_p)$ is defined by $e\mapsto t(e)-i(e)$. We set $H_1(\Gamma;\C_p) = \operatorname{ker}(d)$, these will be called the $1$-cycles.

Now we define the combinatorial object we need. For more general constructions along these lines, see \cite{mikhalkin_zharkov08:tropical_curves,baker_faber11:metric_properties_tropical_abel_jacobi,baker_rabinoff14:skeleton_jacobian}.

\begin{defn}
A \defi{tropical $1$-form}\footnote{Some authors call this a harmonic $1$-cochain.} on $\Gamma$ with values in $\C_p$ is a function $\eta\colon E(\Gamma)\to \C_p$ such that
\begin{enumerate}
    \item $\eta(-e)=-\eta(e)$, and
    \item for each $v\in V(\Gamma)$, we have $\sum_e \eta(e)=0$ where the sum is taken over all edges that are adjacent to $v$ and directed away from $v$ (harmonicity condition).
\end{enumerate} 
Denote the set of such functions by $\Omega^1_{\trop}(\Gamma;\C_p)$.
\end{defn}

Let us consider an example. For an oriented edge $e=vw$ of $\Gamma$, define
\[\eta_e\colon E(\Gamma)\to \C_p, \ \ \ e' \mapsto \begin{cases}
\pm 1 & \text{if}\ e'=\pm e, \\
0 & \text{otherwise}.
\end{cases}\] 
For $C=\sum_e c_e e\in H_1(\Gamma;\C_p)$, the function $\eta_C=\sum_e c_e \eta_e$ is a tropical $1$-form.

\begin{defn}
Let $\eta$ be a tropical $1$-form with values in $\C_p$ and let $\gamma$ be a path in $\Gamma$ specified as a sequence of edges $\gamma=e_1e_2\dots e_\ell$. We define the \defi{tropical integral} of $\eta$ along $\gamma$ by
\[\tint_{\gamma} \eta = \sum_{i=1}^{\ell} \eta(e_i)\in \C_p.\]
\end{defn}

As in \cite[Section~3.10]{katz2020padic}, we may extend the tropical integral to paths between points on $\Gamma$ by parametrizing edges. For a closed path $\gamma$ in $\Gamma$, this integral only depends on $[\gamma]\in H_1(\Gamma;\C_p)$, so it makes sense to write $\ltint_{C} \eta$ for $C\in H_1(\Gamma;\C_p)$.

Note that $\Omega^1_{\trop}(\Gamma;\C_p)\subset C_1(\Gamma;\C_p)^*$. For later use, we provide the following well known proposition which describes the difference.

\begin{prop} \label{harmonicandexact}
We have
\[C_1(\Gamma;\C_p)^* = \Omega^1_{\trop}(\Gamma;\C_p) \oplus \operatorname{Im}(d^*)\] where $d^*\colon C_0(\Gamma;\C_p)^*\to C_1(\Gamma;\C_p)^*$ is the coboundary map.
\end{prop}

\subsubsection*{The Comparison Formula} Recall that we have the curve $X$ over $\bK$. Hereafter, we will abuse notation and use the marker ``an'' to also denote Berkovich analytification; the intention should be clear from the context. In fact, we will sometimes identify a Berkovich analytic space with its corresponding rigid analytic space.

By extending the field of definition if necessary, we assume that the curve $X$ admits a semistable $\mathcal{O}_{\bK}$-model $\fX$. Let $\Gamma\subset X^\an$ and $\tau\colon X^\an\to\Gamma$ be the corresponding skeleton and retraction, respectively. After identifying $\Gamma$ with the dual graph of the special fiber $\fX_{\bk}$, we have the following result whose proof is almost identical to that of \cite[Corollary~3.14]{katz2020padic}.

\begin{prop} \label{l:dualbasis} 
There exists a basis $C_1,\dots,C_h$ of $H_1(\Gamma;\C_p)$ and a basis $\eta_1,\dots,\eta_h$ of $\Omega^1_{\trop}(\Gamma;\C_p)$ with the property that 
\[\tint_{C_i} \eta_j=
\begin{cases}
1 & \text{if }\ i=j,  \\
0 & \text{if }\ i\neq j.
\end{cases}\]
\end{prop} 

Finally, we put everything together.

\begin{thm} \label{mainthm} 
With the above notation, pick a loop $\gamma_i$ in $X^\an$ satisfying $\tau(\gamma_i)=C_i$ for each $i=1,\dots,h$. Let $\omega$ be a meromorphic $1$-form on $X$, let $x,y\in X(\bK)$ and pick a path $\gamma$ in $X^\an$ with $\gamma(0)=x$, $\gamma(1)=y$. Then
\[\Vint_x^y\omega = \BCint_\gamma\omega-\sum_{i=1}^h \left(\BCint_{\gamma_i} \omega\right)\left(\tint_{\tau(\gamma)} \eta_i\right).\]
\end{thm}

The proof of this will occupy the rest of this section. Before giving the proof, we make a few preliminary remarks.

\begin{rem} The path $\gamma$ exists because
the assumption that $X$ is geometrically connected implies that $X^\an$ is path-connected; and the integral $\lBCint_{\gamma_i} \omega$ is well-defined as $\gamma_i$ is unique up to fixed endpoint homotopy.
\end{rem}

\begin{rem}
The Vologodsky and Berkovich--Coleman integrals coincide when $\Gamma$ is a tree, or equivalently, when the Jacobian of $X$ has (potentially) good reduction; and, in particular, when $X$ has (potentially) good reduction.
\end{rem}

\begin{rem}
Theorem~\ref{mainthm} generalizes \cite[Theorem~3.16]{katz2020padic} since, when the integrand is holomorphic, the Vologodsky integral is the same as the abelian integral.
\end{rem}

Theorem~\ref{mainthm} follows from the main result of Besser and Zerbes \cite{BesserZerbes}, which we briefly recall here.

Write $\fX_{\bk} = \cup_i T_i$. By blowing up if necessary, we may assume that every irreducible component is smooth and that two different components intersect at at most one point.

The reduction map $\text{red}\colon X\to \fX_{\bk}$ allows us to cover $X$ by basic wide open spaces: define $U_v = \text{red}^{-1}T_v$ for each $v\in V(\Gamma)$. These spaces intersect along annuli corresponding bijectively to the quotient set of unoriented edges $E(\Gamma)/\pm$. An orientation of an annulus fixes a sign for the residue along this annulus. We make a bijection between oriented edges and oriented annuli by choosing the following convention: 
\begin{quotation}
For an edge $e$, the orientation of the corresponding annulus is the one corresponding to it being the annulus end of $T_{i(e)}$.
\end{quotation}
For an edge and for the corresponding annulus, we may use the same notation.

Choose Coleman primitives $F_v$ for $\omega$ on $U_v$ for each $v\in V(\Gamma)$. Notice that, for an oriented edge $e$, the function $F_{t(e)}|_e - F_{i(e)}|_e$ is constant because both Coleman primitives differentiate to $\omega$. This observation gives a map
\[\eta_{\omega}\colon  E(\Gamma)\to \C_p,\ \ \ e\mapsto F_{t(e)}|_e - F_{i(e)}|_e\]
which obviously satisfies $\eta_{\omega}(-e)=-\eta_{\omega}(e)$. By Proposition~\ref{harmonicandexact}, there is a unique, up to a global constant, way of choosing the $F_v$'s in such a way that the map $\eta_{\omega}$ is a tropical $1$-form. With these choices, thanks to \cite[Theorem~1.1]{BesserZerbes}, the following holds.

\begin{thm} \label{beszermainthm}
If $v,w$ are vertices such that $x\in U_v$ and $y\in U_w$, then
\[\Vint_x^y\omega = F_w(y)-F_v(x).\]
\end{thm}

In other words, Vologodsky integration is \emph{locally} given by Coleman primitives. 

In the following proof, we use the fact that Berkovich--Coleman integrals on basic wide open spaces are ordinary Coleman integrals (under the identification of a Berkovich space and its corresponding rigid space).

\begin{proof}[Proof of Theorem~\ref{mainthm}] 
We continue with the above notation. In the case that $v=w$, the formula in Theorem~\ref{mainthm} is the same as the formula in Theorem~\ref{beszermainthm}. Otherwise, if $\tau(\gamma)=e_1e_2\dots e_\ell$, then we can write $\gamma=\gamma^1\gamma^2\dots\gamma^{\ell+1}$ as a concatenation of paths, each staying in a basic wide open space. Set
\[P_i = \gamma^i(1)=\gamma^{i+1}(0), \ \ \ i=1,2,\dots,\ell.\]
Then we get
\[F_w(y)-F_v(x)
= \BCint_{x}^{P_1}\omega + \BCint_{P_1}^{P_2}\omega + \dots + \BCint_{P_\ell}^{y}\omega + \sum_{i=1}^{\ell} \eta_{\omega}(e_i)
= \BCint_\gamma\omega + \eta_{\omega}(\tau(\gamma)).\] Therefore, we need to show that
\[\eta_{\omega}(\tau(\gamma)) = -\sum_{i=1}^h \left(\BCint_{\gamma_i} \omega\right)\left(\tint_{\tau(\gamma)} \eta_i\right).\]

Since $\eta_{\omega}$ is a tropical $1$-form, $\eta_{\omega}=\sum_{i=1}^{h} c_i\eta_i$ for some constants $c_i$. The computation
\[\eta_{\omega}(C_j) = \sum_{i=1}^{h} c_i\eta_i(C_j) = \sum_{i=1}^{h} c_i\tint_{C_j} \eta_i = c_j\]
gives
\[\eta_{\omega}(\tau(\gamma)) = \sum_{i=1}^{h} \eta_{\omega}(C_i)\eta_i(\tau(\gamma)) = \sum_{i=1}^{h} \eta_{\omega}(C_i)\left(\tint_{\tau(\gamma)} \eta_i\right).\]
It remains to verify that 
\[\eta_{\omega}(C_i) = -\BCint_{\gamma_i} \omega,\ \ \ i=1,\dots,h.\]
If $C_i = e_{i1}e_{i2}\dots e_{i\ell_i}$, then we can write $\gamma_i=\gamma_i^1\gamma_i^2\dots\gamma_i^{\ell_i}$ as a concatenation of paths, each staying in a basic wide open space. Set
\[P_{ij}=
\begin{cases}
\gamma_i^j(1)=\gamma_i^{j+1}(0) & \text{if }\ j=1,2,\dots,\ell_i-1, \\
\gamma_i^{\ell_i}(1)=\gamma_i^1(0) & \text{if }\ j=\ell_i.
\end{cases}\]
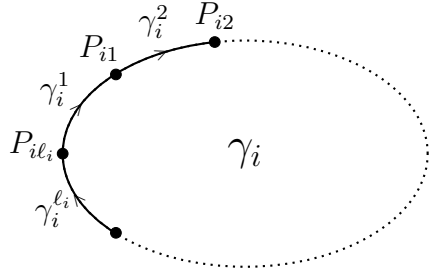
\[\begin{tikzpicture}
\draw [dotted, thick] (-4,0) ellipse (2.4cm and 1.5cm);
\filldraw (-4,0) node{\Large $\gamma_i$};
\draw [thick] (-6.4,0)  arc[x radius = 2.4cm, y radius = 1.5cm, start angle= 180, end angle= 100];
\draw [thick] (-6.4,0)  arc[x radius = 2.4cm, y radius = 1.5cm, start angle= -180, end angle= -135];
\filldraw (-5.7,1.05) circle (2pt);
\filldraw (-5.92,1.40) node{$P_{i1}$};
\filldraw (-4.4,1.48) circle (2pt);
\filldraw (-4.4,1.83) node{$P_{i2}$};
\filldraw (-6.4,0) circle (2pt);
\filldraw (-6.8,0.1) node{$P_{i\ell_i}$};
\filldraw (-5.7,-1.05) circle (2pt);
\node [thick,rotate=50] at (-6.2,0.6) {$\scriptscriptstyle >$};
\filldraw (-6.5,0.85) node{$\gamma_i^1$};
\node [thick,rotate=23] at (-5.1,1.335) {$\scriptscriptstyle >$};
\filldraw (-5.2,1.75) node{$\gamma_i^2$};
\node [thick,rotate=-50] at (-6.2,-0.6) {$\scriptscriptstyle <$};
\filldraw (-6.52,-0.735) node{$\gamma_i^{\ell_i}$};
\end{tikzpicture}\]
By setting $F_{ij} = F_{i(e_{ij})}$ for $j=1,2,\dots,\ell_i$, we see that
\begin{align*}
\BCint_{\gamma_i}\omega &= \BCint_{\gamma_i^1}\omega + \BCint_{\gamma_i^2}\omega + \dots + \BCint_{\gamma_i^{\ell_i}}\omega \\ 
&= \big(F_{i1}(P_{i1}) - F_{i2}(P_{i1})\big) + \big(F_{i2}(P_{i2}) - F_{i3}(P_{i2})\big) + \dots + \big(F_{i\ell_i}(P_{i\ell_i}) - F_{i1}(P_{i\ell_i})\big) \\
&= -\eta_{\omega}(e_{i1})-\eta_{\omega}(e_{i2})- \dots - \eta_{\omega}(e_{i\ell_i}) = -\eta_{\omega}(C_i)
\end{align*} 
as required. 
\end{proof}

\section{Coverings of Curves}
\label{Coverings}

Let $X$ be a smooth, proper and geometrically connected $\bK$-curve. In Section~\ref{ComparisonofTheIntegrals}, in order to express Berkovich--Coleman integrals on $X^\an$ in terms of ordinary Coleman integrals, we covered $X$ by basic wide opens. To do so, we used a semistable model of $X$. Alternatively, we could use a \emph{semistable covering}; an equivalent rigid analytic notion introduced in \cite{CMStableReduction}. Surprisingly, constructing a semistable covering is often easier in practice than constructing a semistable model; and this is what we will do in Section~\ref{NumericalExamples}. In this section, we summarize prior results from \cite[Section~4]{katz2020padic} regarding the construction of semistable coverings of hyperelliptic curves.

\subsection{Semistable Coverings} Let $X$ be a smooth, proper and geometrically connected rigid analytic curve over $\bK$. Let $\cC$ be a semistable covering of $X$ with dual graph $\Gamma(\cC)$, as in \cite[Definition~4.2]{katz2020padic}. We say $\cC$ is \defi{good} with respect to a subset $S\subset X(\C_p)$ if 
\begin{enumerate}
    \item each element of $S$ lies in at most one element of $\cC$, and
    \item for each $U\in\cC$, there exist a curve $X_U$ of good reduction and an embedding $\iota\colon U\to X_U$ such that the points of $\iota(S\cap U(\C_p))$ lie in distinct residue discs.
\end{enumerate}
In this case, we form a new graph $\Gamma(\cC,S)$ from $\Gamma(\cC)$ as follows: we attach half-open edges corresponding to elements of $S$ to the vertices corresponding to the elements of $\cC$ containing them.

\subsection{Hyperelliptic Coverings} \label{HyperellipticCoverings} Now let $X$ be a hyperelliptic curve given by $y^2 = f(x)$ where $f(x)\in \bK[x]$. We write $w\colon X\to X$ for the hyperelliptic involution and $\pi\colon X\to\P^1$ for the hyperelliptic double cover. We define the \defi{roots} of $f(x)$ to be the zeroes of $f(x)$ together with $\infty$ if $f(x)$ has odd degree and write the set of roots as $S_f$. 
Following the methods in \cite[Section~4.3]{katz2020padic}, in particular \cite[Algorithm~1]{katz2020padic}, one can effectively construct a semistable covering $\cC$ of $\P^{1,\an}$ by rational basic wide opens that is good with respect to $S_f$ together with its dual graph $\Gamma(\cC,S_f)$ satisfying the following properties:
\begin{itemize}
    \item The intersection of two distinct elements of $\cC$ is either empty or an annulus.
    \item If $U\in\cC$, then $X_U=\P^{1,\an}$ and the embedding $\iota\colon U\to \P^{1,\an}$ is a fractional linear transformation.
    \item The dual graph $\Gamma(\cC,S_f)$ is a tree and is a graph structure on the skeleton of $\P^{1,\an}\setminus S_f$.
\end{itemize}

The map $\pi\colon X\to\P^1$ will allow us to pass from $\P^{1,\an}$ to $X^\an$. Let $W$ be the set of Weierstrass points, i.e., the set of fixed points of $w$. Then, as explained in \cite[Section~4.9]{katz2020padic}, the set
\[\cD = \{\text{components of }\pi^{-1}(U) \mid U\in\cC\}\]
is a semistable covering of $X^\an$ by hyperelliptic basic wide opens that is good with respect to $W$. This covering and its dual graph $\Gamma(\cD,W)$ enjoy the following properties:
\begin{itemize}
    \item If two distinct elements of $\cD$ are not disjoint, then their intersection is either an annulus or the union of two disjoint annuli.
    \item If $U\in\cD$, then $X_U$ is either the rigid analytification of a rational curve or of a hyperelliptic curve.
    \item The dual graph $\Gamma(\cD,W)$ is a double cover of $\Gamma(\cC,S_f)$ and is a graph structure on the skeleton of $X^\an\setminus W$.
\end{itemize}

\section{Computing Berkovich--Coleman Integrals} 
\label{BC-int-hypcur}

The paper \cite{katz2020padic} describes an effective method for numerically computing Berkovich--Coleman integrals of regular forms on hyperelliptic curves. In this section, we extend the method in a natural way to also cover meromorphic forms. This will make it possible to compute Vologodsky integrals by Theorem~\ref{mainthm}. 

Let $X$ be a hyperelliptic curve given by $y^2 = f(x)$ for some polynomial $f(x)\in \bK[x]$ of degree at least $3$. Recall the coverings $\cC$, $\cD$ constructed in Section~\ref{HyperellipticCoverings}. Here is a rough outline of our method:
\begin{enumerate}
    \item Reduce the problem of computing $\lBCint_{\gamma} \omega$ on $X^\an$ to computing Berkovich--Coleman integrals on certain elements of $\cD$.
    \item For each element $Y$ of $\cD$ of interest, go through the following steps:
    \begin{enumerate}
    \item Identify $Y$ with a basic wide open $Z$ inside the analytification of a curve $\tilde{X}$ of good reduction.
    \item Expand the pull back of the form $\omega|_{Y}$ to $Z$ as a power series in certain meromorphic forms on $\tilde{X}$.
    \item By a pole reduction argument, rewrite the terms in the power series expansion in terms of \emph{basis} elements.
    \item Employ the known integration algorithms on $\tilde{X}$.
    \end{enumerate}
\end{enumerate}

We begin by breaking up Berkovich--Coleman integrals into smaller, more manageable pieces. The main idea has already appeared in the proof of Theorem~\ref{mainthm}.

\subsection{Berkovich--Coleman Integration on Paths} \label{BerkovichColemanIntegrationonPaths} Let $\Gamma$ be the dual graph of $\cD$. For a vertex $v$, let $U_v$ be the corresponding hyperelliptic basic wide open space; for an edge $e=vw$, let $U_e$ be the corresponding component of $U_v\cap U_w$. Pick a point $P_\sigma$ in each $U_\sigma$ for $\sigma = v,e$; we will call these points \textbf{reference points}. We identify $\Gamma$ with the skeleton of $X^\an$ and write $\tau\colon X^\an\to\Gamma$ for the retraction map.

Let $\gamma\colon [0,1]\to X^\an$ be a path (which is allowed to be closed) with ends $x,y$ in $X^\an(\C_p)$. If $x\in U_v$, $y\in U_w$ and $\tau(\gamma)=e_1e_2\dots e_{\ell}$, because the Berkovich--Coleman integral is invariant under fixed endpoint homotopy, we have
\[\BCint_{\gamma} \omega = \BCint_x^{P_v} \omega +
\sum_{i=1}^{\ell} \left(
\BCint_{P_{i(e_i)}}^{P_{e_i}} \omega
+\BCint_{P_{e_i}}^{P_{t(e_i)}} \omega
\right) + 
\BCint_{P_w}^y \omega\]
for every meromorphic $1$-form $\omega$.  Here
\begin{itemize}
    \item the integral from $x$ to $P_v$ is to be performed on $U_v$,
    \item the integral from $P_{i(e_i)}$ to $P_{e_i}$ is to be performed on $U_{i(e_i)}$,
    \item the integral from $P_{e_i}$ to $P_{t(e_i)}$ is to be performed on $U_{t(e_i)}$, and
    \item the integral from $P_w$ to $y$ is to be performed on $U_w$.
\end{itemize}

Now we discuss Berkovich--Coleman integrals on elements of the covering $\cD$. Note that these are ordinary Coleman integrals.

\subsection{Passing to Curves of Good Reduction} \label{sbsctnpassingtogoodred} Let $Y$ be an element of $\cD$. By construction, $Y$ is either the preimage of some element of $\cC$ under $\pi\colon X^\an\to\P^{1,\an}$ or one of the two components of such a preimage. In the latter case, differential forms on $Y$ can be written in terms of $x$ alone and so can be integrated easily; hence we may assume that $Y=\pi^{-1}(U)$ for some $U\in\cC$. We will pass to a curve of good reduction to make use of the existing explicit methods.

We assume for the rest of this section that the polynomial $f(x)$ is monic with integral coefficients in $\bK$. Extend the field $\bK$ so that it contains all finite roots of $f(x)$. We view $U$ as a subset of $\P^{1,\an}$ where we have made a linear fractional transformation to ensure that the elements of $S_f\cap U(\C_p)$, if any, lie in different residue discs. We may suppose without loss of generality that $U$ is $B(0,R)$ minus some closed discs for some $R>1$ and that the set $S_f\cap U(\C_p)$ lies in $\overline{B}(0,1)$. Let $I_1,\dots,I_m$ be the partition of $S_f\cap B(0,R)$ according to which residue disc a point belongs. Relabeling if necessary, we may assume that the number of elements of $I_j$ is
\begin{itemize}
\item at least $2$ and even for $j=1,\dots,k$;
\item at least $2$ and odd for $j=k+1,\dots,\ell$; and
\item $1$ for $j=\ell+1,\dots,m$.
\end{itemize}
Set
\[ L_j =
  \begin{cases}
    \hfil |I_j|/2       &\text{if } j=1,\dots,k;\\
    (|I_j|-1)/2  &\text{if } j=k+1,\dots,\ell.
  \end{cases} \]
For $j=1,\dots,\ell$, pick a point $\beta_j\in \A^1(\bK)\setminus U(\bK)$ in the residue disc containing $I_j$; and for $j=\ell+1,\dots,m$, let $\beta_j$ be the unique element of $I_j$. Put
\[g(x) = \prod_{j=k+1}^{m} (x-\beta_j),\ \ \ h(x) = \prod_{j=1}^{\ell}(x-\beta_j)^{L_j}.\]
Let $I_{\infty}$ denote the set of roots of $f(x)$ lying outside of the open disc $B(0,R)$ and set 
\[k(x) = \left(\prod_{j=1}^{\ell}\prod_{\beta\in I_j}\left(\frac{x-\beta}{x-\beta_j}\right)\right)\left(\prod_{\beta\in I_\infty\setminus\{\infty\}}(x-\beta)\right).\]
By construction,
\begin{itemize}
    \item we have $f(x)=g(x)h(x)^2k(x)$,
    \item the function $\frac{1}{h(x)k(x)^{1/2}}$ is analytic on $U$, and
    \item the polynomial $g(x)$ is non-constant and has at most one root in each residue disc.
\end{itemize}
Then the map
\begin{equation}
\label{IsomBetweenYandZ}
Y=\pi^{-1}(U) \to Z\coloneq\{(x,\tilde{y})\mid \tilde{y}^2=g(x), x\in U\},\ \ \ (x,y) \mapsto \left(x,\frac{y}{h(x)k(x)^{1/2}}\right) 
\end{equation}
is an isomorphism and the curve $\tilde{X}\colon \tilde{y}^2=g(x)$, which is either rational or hyperelliptic, has good reduction. For more details, we refer the reader to the first half of \cite[Section~7]{katz2020padic}. Thanks to this isomorphism, it suffices to compute integrals on the basic wide open space $Z$ inside $\tilde{X}^{\an}$. 

Let $\omega$ be a meromorphic $1$-form on $X$. We will describe what $\omega|_{Y}$ looks like in the new coordinates and expand it in a power series. We start with the following reduction step. 

\subsection{Reduction Step}
For an integer $i$ and an element $\beta$ in $\A^1(\bar{\bK})$, define the differential forms
\[\omega_i = x^i \frac{dx}{2y},\ \ \ \nu_{\beta} = \frac{1}{x-\beta}\frac{dx}{2y}.\]

\begin{prop} 
\label{reductionprop} If $f(x)$ is of degree $d$, then the problem of computing $\lBCint \omega$ reduces to computing
\begin{enumerate}
    \item $\lBCint \omega_i$,\ \ \ $i=0,1,\dots,d-2$; and
    \item $\lBCint \nu_{\beta}$,\ \ \ $\beta$ is a non-root of $f(x).$
\end{enumerate}
\end{prop}

\begin{proof}
We can write $\omega$  as a linear combination, $\omega=\rho + \sum_j d_j \nu_j$, where $\rho$ is of the second kind, $d_j\in \bar{\bK}$ and $\nu_j$ is of the third kind. This gives
\[\BCint \omega = \BCint \rho + \sum_j d_j \BCint \nu_j.\]

In both the odd and even degree cases, the set $\{[\omega_0],[\omega_1],\dots,[\omega_{d-2}]\}$ forms a spanning set for $H_{\dR}^1(X)$. Therefore, we may write $\rho$ as a linear combination of $\omega_0,\omega_1,\dots,\omega_{d-2}$ together with an exact form. On the other hand, recalling the exact sequence \eqref{resdivhom} induced by the residual divisor homomorphism, we may assume that
\[\Res(\nu_j) = (P_j) - (Q_j).\]
Then, up to a holomorphic differential, we have
\[\nu_j = 
\begin{dcases}
\left(\dfrac{y+y(P_j)}{x-x(P_j)}-\dfrac{y+y(Q_j)}{x-x(Q_j)}\right)\dfrac{dx}{2y} &\text{if } P_j \text{ and } Q_j \text{ are finite}; \\
\omit\hfil$\dfrac{y+y(P_j)}{x-x(P_j)}\dfrac{dx}{2y}$\hfil &\text{if } d \text{ is odd},\  P_j \text{ is finite},\ Q_j \text{ is infinite};\\
\omit\hfil$2\omega_g$\hfil &\text{if } d \text{ is even},\  P_j=\infty^{-},\ Q_j=\infty^{+}; \\
\omit\hfil$\dfrac{y+y(P_j)}{x-x(P_j)}\dfrac{dx}{2y}-\omega_g$\hfil &\text{if } d \text{ is even},\  P_j \text{ is finite},\ Q_j=\infty^{-}; \\
\omit\hfil$\dfrac{y+y(P_j)}{x-x(P_j)}\dfrac{dx}{2y}+\omega_g$\hfil &\text{if } d \text{ is even},\  P_j \text{ is finite},\ Q_j=\infty^{+};
\end{dcases}\]
where $\infty^{\pm}$ are the points lying over $\infty$ (in the case when $d$ is even). In any case, the form $\nu_j$ is a linear combination of $\omega_0,\omega_1,\dots,\omega_{g-1},\omega_g$, differentials of the form $\frac{y(P)}{x-x(P)}\frac{dx}{2y}$ and logarithmic differentials. As $\frac{y(P)}{x-x(P)}\frac{dx}{2y}=0$ when $P$ is a Weierstrass point, the claim follows.
\end{proof}

The forms $\{\omega_i\}_{i=0,1,\dots,g-1}$ are already dealt with in \cite{katz2020padic}. For the others, we can proceed in essentially the same way.

\subsection{Power Series Expansion} \label{PowerSeriesExpansion}
Let $i\in\{0,1,\dots,d-2\}$ be an integer, and let $\beta_0$ be an element of $\A^1(\bar{\bK})$ with $f(\beta_0)\neq 0$. We consider the differentials $\omega_i$ and $\nu_{\beta_0}$; they correspond, under the isomorphism \eqref{IsomBetweenYandZ}, to
\[\omega_i(x,\tilde{y}) = \frac{x^i}{h(x)k(x)^{1/2}}\frac{dx}{2\tilde{y}},\ \ \ \nu_{\beta_0}(x,\tilde{y}) = \frac{1}{(x-\beta_0)h(x)k(x)^{1/2}}\frac{dx}{2\tilde{y}}.\]
We will expand these in a power series. Set
\[k_j(x) = \prod_{\beta\in I_j}\Big(1-\frac{\beta-\beta_j}{x-\beta_j}\Big),\ j=1,\dots,\ell;\ \ \ k_\infty(x) = \prod_{\beta\in I_\infty\setminus\{\infty\}}(-\beta)(1-\beta^{-1}x).\]
Thus, we have $k(x)=\big(\prod_{j=1}^{\ell} k_j(x)\big)k_\infty(x)$. Now,
\[\frac{1}{k_j(x)^{1/2}}=\sum_{n=0}^{\infty}\frac{B_{jn}}{(x-\beta_j)^n},\ j=1,\dots,\ell;\ \ \ \frac{1}{k_\infty(x)^{1/2}}=\sum_{n=0}^{\infty}B_{\infty n}x^n\]
for some $B_{jn},B_{\infty n}\in \bar{\bK}$. Then,
\begin{align*}
\frac{1}{h(x)k(x)^{1/2}} &= \left(\prod_{j=1}^\ell \sum_{n=0}^{\infty}\frac{B_{jn}}{(x-\beta_j)^{n+L_j}}\right)\left(\sum_{n=0}^{\infty}B_{\infty n}x^{n}\right)\\ 
&= \sum\limits_{\substack{n_1\geq L_1,\dots,n_{\ell}\geq L_{\ell} \\ n_\infty\geq 0}}\left(B_{n_1,\dots,n_{\ell},n_\infty}x^{n_{\infty}}\prod_{j=1}^{\ell}\frac{1}{(x-\beta_j)^{n_j}}\right)
\end{align*} 
for some $B_{n_1,\dots,n_{\ell},n_\infty}\in \bar{\bK}$, which can be bounded as follows.
\begin{prop} \label{BoundsForB's}
\emph{(\cite[Proposition~7.1]{katz2020padic})}
There exists a constant $C$ satisfying \[\|B_{n_1,\dots,n_{\ell},n_{\infty}}\|_p\leq Cp^{-{\frac{(n_\infty-i)+\sum_{j=1}^{\ell}(n_j-L_j)}{e}}},\]
where $e$ is the ramification degree of $\bK$ over $\Q_p$.
\end{prop}

Therefore, the expressions
\begin{align*}
\omega_i(x,\tilde{y}) &= \sum\limits_{\substack{n_1\geq L_1,\dots,n_{\ell}\geq L_{\ell} \\ n_\infty\geq i}}\left(B_{n_1,\dots,n_{\ell},n_\infty}x^{n_{\infty}}\prod_{j=1}^{\ell}\frac{1}{(x-\beta_j)^{n_j}}\frac{dx}{2\tilde{y}}\right),\\ 
\nu_{\beta_0}(x,\tilde{y}) &= \sum\limits_{\substack{n_1\geq L_1,\dots,n_{\ell}\geq L_{\ell} \\ n_\infty\geq 0}}\left(B_{n_1,\dots,n_{\ell},n_\infty}\frac{x^{n_{\infty}}}{x-\beta_0}\prod_{j=1}^{\ell}\frac{1}{(x-\beta_j)^{n_j}}\frac{dx}{2\tilde{y}}\right),
\end{align*} 
make sense. Since these two expressions are almost identical, we will focus only on the former; the latter can be handled in a completely analogous way.

Clearly, the form $\omega_i(x,\tilde{y})$ does not extend to the complete curve $\tilde{X}$ as a meromorphic differential, hence we can not perform its integral on $\tilde{X}^\an$, at least not directly.

\subsection{Interchanging Integration and Summation} \label{InterchangingIntegrationandSummation} 
Notice that $\lBCint\omega_i(x,\tilde{y})$ is the integral of a series of $1$-forms. Using \cite[Proposition~3.5]{katz2020padic}, we can interchange the order of summation and integration:
\begin{equation}
\label{intsum=sumint}
\BCint\omega_i(x,\tilde{y}) =\sum\limits_{\substack{n_1\geq L_1,\dots,n_{\ell}\geq L_{\ell} \\ n_\infty\geq i}}\left(B_{n_1,\dots,n_{\ell},n_\infty}\BCint x^{n_{\infty}}\prod_{j=1}^{\ell}\frac{1}{(x-\beta_j)^{n_j}}\frac{dx}{2\tilde{y}}\right).   
\end{equation}
The proof of this fact involves a careful analysis of bounding the integrals of terms. See \cite[Sections~6, 7]{katz2020padic} for details.

By \eqref{intsum=sumint}, we need to integrate
the terms in the series expansion of  $\omega_i(x,\tilde{y})$. Let
\[\eta = \eta_{n_1,\dots,n_{\ell},n_{\infty}} = x^{n_{\infty}}\prod_{j=1}^{\ell}\frac{1}{(x-\beta_j)^{n_j}}\frac{dx}{2\tilde{y}}\]
be such a term. As $\eta$ can be seen as a meromorphic form on the complete curve $\tilde{X}$, we can perform its integral on $\tilde{X}^{\an}$. First, we need to express this differential in terms of our \emph{basis} elements.

\subsection{Pole Reduction} Set $\tilde{d} = \text{deg}(g(x))$. Following the methods in \cite[Section~6]{katz2020padic}, in particular \cite[Algorithms~2, 3, 4, 5]{katz2020padic}, we can effectively find a meromorphic function $F = F(x,\tilde{y})$ and constants $c_i,d_j$ such that $\eta$ can be written as
\begin{equation}
\label{etainbasis}
\eta = dF + \sum_{i=0}^{\tilde{d}-2} c_i\tilde{\omega}_i + \sum_{j=1}^k d_j\tilde{\nu}_j
\end{equation}
where $k$ is as in Section~\ref{sbsctnpassingtogoodred} and
\[\tilde{\omega}_i = x^i \frac{dx}{2\tilde{y}},\ \ \ \tilde{\nu}_j = \frac{1}{x-\beta_j}\frac{dx}{2\tilde{y}}.\]
We sketch the procedure. It consists of two parts, reducing the pole order at finite points and at the point(s) at infinity, respectively.

We start with finite points. For an element $\beta$ in $\A^1(\bar{\bK})$ and a positive integer $m$, consider the exact form 
\[\mu_{\beta,m} = d\left(\frac{\tilde{y}}{(x-\beta)^m}\right) = \frac{(x-\beta)g'(x)-2mg(x)}{(x-\beta)^{m+1}}\frac{dx}{2\tilde{y}}\]
which has poles at the point(s) above $\beta$ and possibly also at the point(s) at infinity. If $g(\beta)\neq 0$, there are two poles of order $m+1$ above $\beta$; if $g(\beta)=0$, there is a single pole above $\beta$ and its order is $2m$. By subtracting off a linear combination of the forms $\mu_{\beta,m}$, we can cancel the non-simple poles of $\eta$ at non-Weierstrass points and the poles of $\eta$ at Weierstrass points. Then, by subtracting off suitable multiples of $\tilde{\nu}_j$, we can cancel the simple poles at non-Weierstrass points. Let $\eta'$ be the remainder.

We now move on to the point(s) at infinity. Notice that $\eta'$ is of the form $p(x)\frac{dx}{2\tilde{y}}$ for some polynomial $p(x)$. We need to lower the degree of $p(x)$. For a non-negative integer $m$, consider
\[\mu_{\infty,m} = d(x^m\tilde{y}) = \big(x^mg'(x)+2 mx^{m-1}g(x)\big)\frac{dx}{2\tilde{y}}\]
which has poles at the point(s) at infinity. We can subtract an appropriate linear combination of the differentials $\mu_{\infty,m}$ from $\eta'$ so that the remainder is of the form $q(x)\frac{dx}{2\tilde{y}}$ for some polynomial $q(x)$ of degree at most $\tilde{d}-2$, and hence can be expressed in terms of the differentials $\tilde{\omega}_i$.

We may bound the constants appearing in \eqref{etainbasis} as follows.

\begin{prop} 
\label{boundsoncoefficients}
Write
\[\eta = x^{n_{\infty}}\prod_{j=1}^{\ell}\frac{1}{(x-\beta_j)^{n_j}}\frac{dx}{2\tilde{y}} =dF + \sum_{i=0}^{\tilde{d}-2} c_i\tilde{\omega}_i + \sum_{j=1}^k d_j\tilde{\nu}_j\]
as above. Then we have the following:
\begin{enumerate}
    \item \label{boundsoncoefficientsF} The function $F$ is of the form $F = F_{nw}+F_w+F_\infty$ where
    \begin{enumerate}
        \item the absolute values of the coefficients of $F_{nw}$ are at most
        \[\max_{i=1,\dots,k}\left(p^{n_i/(p-1)}\right),\]
        \item the absolute values of the coefficients of $F_{w}$ are at most
        \[\max_{i=k+1,\dots,\ell}\left(n_ip^{1+n_i/(p-1)}\right),\]
        \item the absolute values of the coefficients of $F_{\infty}$ are at most the maximum of the following:
        \begin{enumerate}
            \item $\max_{i=1,\dots,k}\left(p^{n_i/(p-1)}\right)$,
            \item $\max_{i=k+1,\dots,\ell}\left(n_ip^{1+n_i/(p-1)}\right)$,
            \item $\tilde{d}(\tilde{d}+n_\infty)p^{2+n_\infty/(p-1)}$.
        \end{enumerate}
    \end{enumerate}
    \item \label{boundsoncoefficientsc} The absolute values of the $c_i$'s are at most
    \[\max_{i=1,\dots,\ell}(n_i,\tilde{d}(\tilde{d}+n_\infty))p^{2+\max_{i=1,\dots,\ell}(n_i/(p-1),n_\infty/(p-1))}.\]
    \item \label{boundsoncoefficientsd} The absolute values of the $d_j$'s are at most $1$.
\end{enumerate}
\end{prop}

\begin{proof} This follows easily from \cite[Proposition~6.9]{katz2020padic} and \cite[Proposition~7.2]{katz2020padic}.
\end{proof}

Finally, the known integration algorithms come into play.

\subsection{Integration Algorithms} \label{IntegrationAlgorithms} For $R,S\in Z(\C_p)$ at which the form $\eta$ is regular, the equality \eqref{etainbasis} gives
\[\BCint_S^R\eta=F(R)-F(S)+\sum_{i=0}^{\tilde{d}-2} c_i\BCint_S^R\tilde{\omega}_i + \sum_{j=1}^k d_j\BCint_S^R\tilde{\nu}_j.\]
If the polynomial $g(x)$ is of degree at most $2$, the curve $\tilde{X}$ can be parameterized and hence $\eta$ can be integrated easily. Otherwise, we compute the integrals $\lBCint \tilde{\omega}_i$ (resp. $\lBCint \tilde{\nu}_j$) as described in \cite{BBKExplicit,balakrishnan2015coleman} (resp. \cite{BBComputing,balakrishnan2016quadratic}). Beware however that these papers have certain restrictions on the base field $\bK$ and the parity of the degree $\tilde{d}$. See \cite[Section~5.2]{katz2020padic} for a summary of integration algorithms as well as for the mentioned restrictions.

We record, for later use, the following proposition; it gives a bound for the absolute value of $\lBCint_S^R\eta$ in a special case.

\begin{prop}
\label{boundintegralofeta}
Suppose that $R$ and $S$ lie in finite non-Weierstrass residue discs. Suppose further that these discs are different from any of the residue discs of $\beta_1,\dots,\beta_k$. Then the value $\left\|\lBCint_S^R\eta\right\|_p$ is at most the maximum of the following:
\begin{enumerate}
    \item $\max_{i=1,\dots,k}\left(p^{n_i/(p-1)}\right)$,
    \item $\max_{i=k+1,\dots,\ell}\left(n_ip^{1+n_i/(p-1)}\right)$,
    \item $\tilde{d}(\tilde{d}+n_\infty)p^{2+n_\infty/(p-1)}$,
    \item \label{dominatingpart} $\max_{i=1,\dots,\ell}(n_i,\tilde{d}(\tilde{d}+n_\infty)) p^{2+\max_{i=1,\dots,\ell}(n_i/(p-1),n_\infty/(p-1))} \max_{i}\left(\left\|\lBCint_S^R\tilde{\omega}_i\right\|_p\right)$,
    \item $\max_{j}\left(\left\|\lBCint_S^R\tilde{\nu}_j\right\|_p\right)$.
\end{enumerate}

\end{prop}

\begin{proof} 
Let $P$ be one of the points $R,S$.
Thanks to \eqref{boundsoncoefficientsc} and \eqref{boundsoncoefficientsd} in Proposition~\ref{boundsoncoefficients}, it suffices to show that
\[\left\|F(P)\right\|_p\leq\max\left(\max_{i=1,\dots,k}\left(p^{n_i/(p-1)}\right), \max_{i=k+1,\dots,\ell}\left(n_ip^{1+n_i/(p-1)}\right), \tilde{d}(\tilde{d}+n_\infty)p^{2+n_\infty/(p-1)}\right).\]
But this follows from \eqref{boundsoncoefficientsF} in Proposition~\ref{boundsoncoefficients} together with the observation that our assumptions imply that
\[\left\|\frac{\tilde{y}(P)}{x(P)-\beta_i}\right\|_p = 1\]
for each $i=1,\dots,\ell$.
\end{proof}

\subsection{Error Bounds} 
\label{Error Bounds}
Let $R$ and $S$ be two points in $Z(\C_p)$ at which the form $\omega_i(x,\tilde{y})$ is regular. The equality \eqref{intsum=sumint} gives 
\[\BCint_{S}^{R}\omega_i(x,\tilde{y}) =\sum_{(n_1,\dots,n_{\ell},n_\infty)\in I}\left(B_{n_1,\dots,n_{\ell},n_\infty}\BCint_{S}^{R} x^{n_{\infty}}\prod_{j=1}^{\ell}\frac{1}{(x-\beta_j)^{n_j}}\frac{dx}{2\tilde{y}}\right)\]
where
\[I = \{(n_1,\dots,n_{\ell},n_\infty)\in\Z^{\ell+1} \mid n_1\geq L_1,\dots,n_{\ell}\geq L_{\ell}, n_\infty\geq i\}.\]
This is an infinite sum, but in practice, we choose a sufficiently large positive integer $n$ and compute only the finite sum
\[\sum_{(n_1,\dots,n_{\ell},n_\infty)\in I_{<n}}\left(B_{n_1,\dots,n_{\ell},n_\infty}\BCint_{S}^{R} x^{n_{\infty}}\prod_{j=1}^{\ell}\frac{1}{(x-\beta_j)^{n_j}}\frac{dx}{2\tilde{y}}\right),\]
where
\[I_{<n} = \{(n_1,\dots,n_{\ell},n_\infty)\in I \mid (n_\infty+\textstyle\sum_{j=1}^{\ell} n_j) - (i + \textstyle\sum_{j=1}^{\ell} L_j) < n\}.\]
In this subsection, we bound the error introduced by omitting other terms; this gives an idea how quickly the resulting integral converges to the required integral.

We only discuss an important special case, namely the case where the endpoints $R$ and $S$ lie in finite non-Weierstrass residue discs; the other cases can be treated similarly. Recall that $e$ is the ramification degree of $\bK$ over $\Q_p$.

\begin{prop} 
Suppose that $R$ and $S$ lie in finite non-Weierstrass residue discs, and that their discs are distinct from any of the discs of $\beta_1,\dots,\beta_k$. Suppose further that $e<p-1$ and set $r = \frac{1}{e} - \frac{1}{p-1}$. Then there exists a constant $D$ such that for every sufficiently large positive integer $n$, the error term
\[\sum_{(n_1,\dots,n_{\ell},n_\infty)\in I_{\geq n}}\left(B_{n_1,\dots,n_{\ell},n_\infty}\BCint_{S}^{R} x^{n_{\infty}}\prod_{j=1}^{\ell}\frac{1}{(x-\beta_j)^{n_j}}\frac{dx}{2\tilde{y}}\right),\]
where 
\[I_{\geq n} = \{(n_1,\dots,n_{\ell},n_\infty)\in I \mid (n_\infty+\textstyle\sum_{j=1}^{\ell} n_j) - (i + \textstyle\sum_{j=1}^{\ell} L_j) \geq n\},\]
has norm at most
\[Dp^{(-r/2)n}.\]
\end{prop}

\begin{proof}
Let $(n_1,\dots,n_{\ell},n_\infty)\in I_{\geq n}$ and set 
\[B = B_{n_1,\dots,n_{\ell},n_{\infty}},\ \ \ \eta = \eta_{n_1,\dots,n_{\ell},n_{\infty}} = x^{n_{\infty}}\prod_{j=1}^{\ell}\frac{1}{(x-\beta_j)^{n_j}}\frac{dx}{2\tilde{y}}.\]
We need to bound $\left\|B\lBCint_S^R\eta\right\|_p$ in terms of $n$. Write $N = n_\infty+\sum_{j=1}^{\ell} n_j$ and $L = i + \sum_{j=1}^{\ell} L_j$; then $N$ increases as $n$ increases, since $N\geq L + n$.

By Proposition~\ref{BoundsForB's}, there is a constant $C$ such that
\[\|B\|_p\leq C p^{L/e}p^{-N/e}.\]
According to Proposition~\ref{boundintegralofeta}, we have
\[\left\|\BCint_S^R\eta\right\|_p\leq \max_{i=1,\dots,\ell}(n_i,\tilde{d}(\tilde{d}+n_\infty)) p^{2+\max_{i=1,\dots,\ell}(n_i/(p-1),n_\infty/(p-1))} \textstyle \max_{i}\left(\left\|\lBCint_S^R\tilde{\omega}_i\right\|_p\right)\]
for sufficiently large $n$. The inequalities
\begin{align*}
p^{2+\max_{i=1,\dots,\ell}(n_i/(p-1),n_\infty/(p-1))} &\leq p^2p^{N/(p-1)} \\ 
\max_{i=1,\dots,\ell}(n_i,\tilde{d}(\tilde{d}+n_\infty)) &\leq \tilde{d}^2N 
\end{align*} 
hold by definition of $N$. On the other hand, the inequality
\[N \leq p^{(r/2)N}\]
also holds when $n$ is sufficiently large. Combining these inequalities gives
\begin{align*}
\left\|B\BCint_S^R\eta\right\|_p &\leq C\tilde{d}^2p^{2+L/e}\textstyle \max_{i}\left(\left\|\lBCint_S^R\tilde{\omega}_i\right\|_p\right)p^{(-r/2)N}\\
 &\leq C\tilde{d}^2p^{2 + L/e - (r/2)L}\textstyle \max_{i}\left(\left\|\lBCint_S^R\tilde{\omega}_i\right\|_p\right)p^{(-r/2)n}
\end{align*}
for sufficiently large $n$; the proposition follows.
\end{proof}

Two remarks are in order. We keep the notation introduced so far.

\begin{rem}
Let $n$ be any positive integer with the following property: for each tuple $(n_1,\dots,n_{\ell},n_\infty)$ in $I_{\geq n}$, 
\begin{enumerate}
    \item the term \eqref{dominatingpart} in Proposition~\ref{boundintegralofeta} dominates the others, and
    \item we have $N \leq p^{(r/2)N}$ (recall that $N = n_\infty+\sum_{j=1}^{\ell} n_j$).
\end{enumerate}
Then $n$ is sufficiently large for our purposes.
\end{rem}

\begin{rem}
By construction, we have 
\[D = C\tilde{d}^2p^{2 + L/e - (r/2)L}\max_{i}\left(\left\|\BCint_S^R\tilde{\omega}_i\right\|_p\right).\]
One can try to make the constant $C$ (and therefore $D$) explicit by analyzing its construction more closely.
\end{rem}

\section{The Algorithm}
\label{TheAlgorithm}

In this section, we use the material from the previous sections to give an algorithm for computing Vologodsky integrals on hyperelliptic curves.

Let $X$ be a hyperelliptic curve given by $y^2 = f(x)$ for some monic polynomial $f(x)$ with coefficients in $\mathcal{O}_{\bK}$. Let $\cD$ be the semistable covering of $X^\an$ constructed in Section~\ref{HyperellipticCoverings} and let $\Gamma$ be its dual graph. As before, we write $\tau\colon X^\an\to\Gamma$ for the retraction map by identifying $\Gamma$ with the skeleton of $X^\an$.

\smallskip

\begin{algorithm}[H] \label{integrationalgorithm}
\caption{\bf Computing Vologodsky integrals}
\KwIn{
\begin{itemize}
    \item A meromorphic $1$-form $\omega$ on $X$.
    \item Points $x,y\in X(\bK)$.
\end{itemize}
}
\KwOut{The integral $\lVint_x^y \omega$.}

\begin{enumerate}
    \item Pick a path $\gamma$ in $X^\an$ from $x$ to $y$ and compute the Berkovich--Coleman integral
    \[\lBCint_{\gamma} \omega\] 
    as in Section~\ref{BC-int-hypcur}.
    \item Determine a basis $C_1,\dots,C_h$ of $H_1(\Gamma;\C_p)$ and a basis $\eta_1,\dots,\eta_h$ of $\Omega^1_{\trop}(\Gamma;\C_p)$ such that 
    \[\tint_{C_i} \eta_j=
    \begin{cases}
    1 & \text{if }\ i=j,  \\
    0 & \text{if }\ i\neq j.
    \end{cases}\]
    \item For each $i$, pick a loop $\gamma_i$ in $X^{\an}$ with the property that $\tau(\gamma_i)=C_i$; and compute the Berkovich--Coleman periods
    \[\BCint_{\gamma_i} \omega,\ \ \ i=1,\dots,h\]
    as in Section~\ref{BC-int-hypcur}.
    \item Compute the tropical integrals 
    \[\tint_{\tau(\gamma)} \eta_i,\ \ \ i=1,\dots,h.\]
    \item Return
    \[\Vint_x^y \omega = \BCint_{\gamma}\omega-\sum_{i=1}^h \left(\BCint_{\gamma_i} \omega\right)\left(\tint_{\tau(\gamma)} \eta_i\right).\]
\end{enumerate} 
\end{algorithm}

\smallskip

In practice, computing tropical integrals is quite easy. It might happen that $\ltint_{\tau(\gamma)} \eta_i = 0$ for some $i$, in which case there is no need to compute $\lBCint_{\gamma_i} \omega$ of course. 

\section{Numerical Examples}
\label{NumericalExamples}

In this final section, we provide two examples computed in Sage \cite{sagemath}, with assistance from Magma \cite{MR1484478}; see also the examples in \cite[Section~9]{katz2020padic}. Let us start with the following remarks.

\begin{itemize}
    \item \emph{Sage restriction.} A Vologodsky integral on a hyperelliptic curve $X/\Q_p$ with endpoints in $X(\Q_p)$ is an element of $\Q_p$. In our approach, we express such an integral as a sum of Berkovich--Coleman integrals, each of which lies in a possibly different finite extension. Indeed, taking square roots might force us to work with unramified extensions and reference points corresponding to edges might lie in ramified extensions. In Sage, one can define unramified extensions and Eisenstein extensions individually, however, conversion between these extensions has not been implemented yet. To deal with this restriction, all computations will take place in a single extension in each of our examples.
    \item \emph{Branch of logarithm.} As discussed in Section~\ref{integrationsection}, the Vologodsky and Berkovich--Coleman integrals require a branch of the $p$-adic logarithm. We choose the branch that takes the value $0$ at $p$.
\end{itemize}

In the examples below, $\omega_i$ will denote the differential $x^i\frac{dx}{2y}$ on the corresponding hyperelliptic curve.

\begin{eg} \label{ex1}
Consider the elliptic curve $X/\Q$ \cite[\href{https://www.lmfdb.org/EllipticCurve/Q/6622/i/3}{6622.i3}]{lmfdb} given by 
\[y^2 = f(x) = x^3-1351755x+555015942\]
which has split multiplicative reduction at the prime $p = 43$. In this example, we will compute the $p$-part of the (extended) Coleman--Gross $p$-adic height pairing on $X$, which is given in terms of a Vologodsky integral. We first review (a very simplified version of) the definition, referring the reader to  the beginning of \cite[Section~2]{BesserPairing} for details.

Let $P$ and $R$ be points of $X$ such that $P,-P,R,-R$ are pairwise distinct. The Coleman--Gross $p$-adic height pairing is, by definition, a sum of local terms
\[h(P,R) = \sum_v h_v(P,R)\]
over all prime numbers $v$. The local components away from $p$ are described using arithmetic intersection theory, see \cite[(1.3)]{coleman1989} for a more precise formulation. On the other hand, if $\Psi$ is the map to $H_{\dR}^1(X/\Q_p)$ defined in \cite[Proposition~2.5]{coleman1989}, then there exists a unique form $\omega_P$ of the third kind such that
\[\Res(\omega_P) = (P)-(-P),\ \ \ \Psi(\omega_P)\in \langle [\omega_1] \rangle\footnote{In fact, we can replace $\langle [\omega_1] \rangle$ by any subspace $W\subset H_{\dR}^1(X/\Q_p)$ complementary to the space of holomorphic forms. Our choice is quite far from being \emph{canonical}.}\]
and the local height pairing at $p$ is defined via
\[h_p(P,R) = \Vint_{-R}^{R} \omega_P.\]

The local heights away from $p$ behave in much the same way as local archimedean heights and can be computed as explained in \cite[Section~3.1]{balakrishnan2016p}; see \cite{holmes2012computing,muller2014computing,van2020explicit} for a detailed account. On the other hand, one can check, using the reformulation of $\Psi$ in \cite[Section~3]{besser2005p} in terms of \emph{local} and \emph{global indices}, that the form $\omega_P$ is nothing but
\[\frac{y(P)}{x-x(P)}\frac{dx}{y} - \left(\Vint_{P}^{-P} \omega_1\right)\omega_0\]
which gives
\[h_p(P,R) = \Vint_{-R}^{R} \frac{y(P)}{x-x(P)}\frac{dx}{y} + \Vint_{-R}^{R} \omega_0\Vint_{-P}^{P} \omega_1.\]
Our techniques allow us to compute the integrals on the right-hand side, hence the local component at $p$.

As a concrete example, let $P=(-501,33264), R=(219,16416)$. The polynomial $f(x)$ factors as
\[f(x) = (x-507)(x-\beta_{+})(x-\beta_{-}),\ \ \ \beta_{\pm} = -\frac{3}{2}(169\pm 33\sqrt{473})\]
and the set $\{U_1,U_2\}$ is a semistable covering of $\P^{1,\an}$ that is good with respect to $S_f=\{507,\beta_{\pm},\infty\}$ where 
\[U_1 = \P^{1,\an}\setminus\overline{B}(26,1/\sqrt{43}),\ \ \ U_2 = B(26,1).\] 
The corresponding dual graph is as follows:
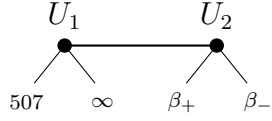
\[\begin{tikzpicture}
\draw [thick] (0,0) -- (2,0);
\filldraw (0,0) circle (2.5pt);
\filldraw (0,0.4) node{$U_1$};
\filldraw (2,0) circle (2.5pt);
\filldraw (2,0.4) node{$U_2$};
\draw (0,0) -- (-0.4,-0.5);
\draw (0,0) -- (0.4,-0.5);
\draw (2,0) -- (1.6,-0.5);
\draw (2,0) -- (2.4,-0.5);
\filldraw (-0.5,-0.75) node{\tiny{$507$}};
\filldraw (0.5,-0.75) node{\tiny{$\infty$}};
\filldraw (1.55,-0.75) node{\tiny{$\beta_{+}$}};
\filldraw (2.55,-0.75) node{\tiny{$\beta_{-}$}};
\end{tikzpicture}\]
Both $P$ and $R$ lie in the component $\pi^{-1}(U_1)$, which is embedded into a rational curve:
\[\pi^{-1}(U_1) \simeq \{(x,\tilde{y})\mid\tilde{y}^2=x-507,\ x\in U_1\}\] 
where 
\[\tilde{y}=\frac{y}{\ell(x)},\ \ \ \ell(x) = \Big(x+\frac{507}{2}\Big)\Big(1-\frac{4635873/4}{(x+507/2)^2}\Big)^{1/2}.\] 
Our computations give
\begin{align*}
\Vint_{-R}^{R} \frac{y(P)}{x-x(P)}\frac{dx}{y} &= 29 \cdot 43 + 29 \cdot 43^2 + 18 \cdot 43^3 + 29 \cdot 43^4 + 3 \cdot 43^5 + O(43^6),\\
\Vint_{-R}^{R} \omega_0 &= 12 \cdot 43^2 + 43^3 + 18 \cdot 43^4 + 40 \cdot 43^5 + O(43^6),\\
\Vint_{-P}^{P} \omega_1 &= 25 + 11 \cdot 43 + 34 \cdot 43^2 + 26 \cdot 43^3 + 25 \cdot 43^4 + 34 \cdot 43^5 + O(43^6),
\end{align*}
from which we get
\[h_p(P,R) = 29 \cdot 43 + 28 \cdot 43^2 + 10 \cdot 43^3 + 42 \cdot 43^4 + 19 \cdot 43^5 + O(43^6).\]

Note that the subspace spanned by $[\omega_1]$ of $H_{\dR}^1(X/\Q_p)$ is isotropic with respect to the cup product pairing; hence the local height pairing at $p$ must be symmetric. As a consistency check, we also compute $h_p(R,P)$. In this case, we have
\begin{align*}
\Vint_{-P}^{P} \frac{y(R)}{x-x(R)}\frac{dx}{y} &= 29 \cdot 43 + 21 \cdot 43^2 + 35 \cdot 43^3 + 20 \cdot 43^4 + 10 \cdot 43^5 + O(43^6),\\
\Vint_{-P}^{P} \omega_0 &= 12 \cdot 43^2 + 43^3 + 18 \cdot 43^4 + 40 \cdot 43^5 + O(43^6),\\
\Vint_{-R}^{R} \omega_1 &= 40 + 8 \cdot 43 + 34 \cdot 43^2 + 26 \cdot 43^3 + 25 \cdot 43^4 + 34 \cdot 43^5 + O(43^6),
\end{align*}
which give
\[h_p(R,P) = 29 \cdot 43 + 28 \cdot 43^2 + 10 \cdot 43^3 + 42 \cdot 43^4 + 19 \cdot 43^5 + O(43^6),\]
demonstrating the symmetry of $h_p$.

We end this example with another consistency check concerning the \emph{global} picture. An important feature of the Coleman--Gross $p$-adic height pairing is that it vanishes if one of the points is torsion. In order to observe this numerically, let $P=(379,9856)$ and $R = (-501,33264)$. We compute the $p$-part as
\[h_p(P,R) = 43 + 21 \cdot 43^{2} + 28 \cdot 43^{3} + 25 \cdot 43^{4} + 3 \cdot 43^{5} + O(43^{6}).\]
Away from $p$, using the Magma implementation of the algorithm developed in van Bommel--Holmes--M{\"u}ller \cite{van2020explicit}, we have
\begin{align*}
\sum_{v\neq p} h_v(P,R) &= -\frac{2}{3}\Log(2) + 2\Log(5) - \frac{2}{3}\Log(11)\\
&= 42 \cdot 43 + 21 \cdot 43^{2} + 14 \cdot 43^{3} + 17 \cdot 43^{4} + 39 \cdot 43^{5} + O(43^{6}).
\end{align*}
Combining these two results, we get
\[h(P,R) = O(43^{6})\]
which is consistent with the fact that $R$ is torsion.
\end{eg}

\begin{eg} \label{ex2} Consider the hyperelliptic curve $X/\Q$ \cite[\href{https://www.lmfdb.org/Genus2Curve/Q/3950/b/39500/1}{3950.b.39500.1}]{lmfdb} given by
\[y^2 = f(x) = (x^2-x-1)(x^4+x^3-6x^2+5x-5).\]
According to the database, the Mordell-Weil group of the Jacobian of $X$ over $\Q$ is isomorphic to $\Z/12\Z$. Therefore, for any two points $R,S\in X(\Q)$, the Vologodsky integrals of holomorphic forms against the divisor $(R)-(S)$ must vanish. In this example, we will compute the integrals of $\omega_0,\omega_1,\omega_2,\omega_3$ and $\omega_4$ between two rational points and observe this vanishing numerically.

Note that $p=5$ is a prime of bad reduction for $X$. Moreover, the corresponding (stable) reduction is a genus $2$ \emph{banana} curve, i.e., the union of two projective lines meeting transversally at three points, as represented in the following figure:
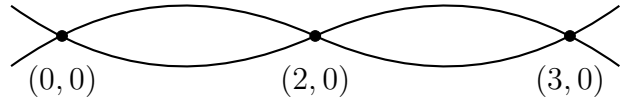
\[\begin{tikzpicture}
\draw[thick] (0,-0.4) to[bend left] (4,0);
\draw[thick] (4,0) to[bend right] (8,0.4);
\draw[thick] (0,0.4) to[bend right] (4,0);
\draw[thick] (4,0) to[bend left] (8,-0.4);
\filldraw (0.67,0) circle (2pt);
\filldraw (4,0) circle (2pt);
\filldraw (7.35,0) circle (2pt);
\filldraw (0.67,-0.6) node{$(0,0)$};
\filldraw (4,-0.6) node{$(2,0)$};
\filldraw (7.35,-0.6) node{$(3,0)$};
\end{tikzpicture}\]
Hereafter, we consider $X$ over the field $\Q_5$. Then the polynomial $f(x)$ factors as the product of three quadratic monic polynomials:
\[f(x) = f_1(x)f_2(x)f_3(x),\ \ \ f_j(x) =x^2+A_jx+B_j\in\Q_5[x].\]
Relabeling if necessary, we may assume that
\[f_1(x)\equiv x^2,\ \ \ f_2(x)\equiv (x-2)^2,\ \ \ f_3(x)\equiv (x-3)^2\ (\operatorname{mod} 5).\]

We begin by constructing coverings. The set $\cC = \{U,U_1,U_2,U_3\}$ is a semistable covering of $\P^{1,\an}$ that is good with respect to the roots of $f(x)$ where
\begin{align*}
U\ &= \P^{1,\an}\setminus\big(\overline{B}(0,1/\sqrt{5})\cup \overline{B}(2,1/\sqrt{5})\cup \overline{B}(3,1/\sqrt{5})\big), \\
U_1 &= B(0,1), \\
U_2 &= B(2,1), \\
U_3 &= B(3,1).
\end{align*}
Set
\[\ell_j(x) = \Big(x+\frac{A_j}{2}\Big)\left(1+\frac{B_j-A_j^2/4}{(x+A_j/2)^2}\right)^{1/2},\ \ \ j=1,2,3.\]
Then $\ell_j(x)^2 = f_j(x)$ for $x$ in the domain of convergence. For the first component, we have
\[\pi^{-1}(U) \simeq \{(x,\tilde{y})\mid\tilde{y}^2=1,\ x\in U\} = \{(x,\pm 1)\mid x\in U\}\] 
where 
\[\tilde{y}=\frac{y}{\ell(x)},\ \ \ \ell(x) = \ell_1(x)\ell_2(x)\ell_3(x).\]
Let $v_+$ (resp. $v_-$) denote the component of $\pi^{-1}(U)$ corresponding to $\{(x,1)\mid x\in U\}$ (resp. $\{(x,-1)\mid x\in U\}$). For the other components, we have
\[v_j\coloneq \pi^{-1}(U_j) \simeq \{(x,\tilde{y})\mid\tilde{y}^2=f_j(x),\ x\in U_j\},\ \ \ j=1,2,3\] 
where 
\[\tilde{y}=\frac{y}{\ell(x)},\ \ \ \ell(x) =
\begin{cases}
\ell_2(x)\ell_3(x) &\text{if } j=1,\\
\ell_1(x)\ell_3(x) &\text{if } j=2,\\
\ell_1(x)\ell_2(x) &\text{if } j=3.\\
\end{cases}\] 
Consequently, $\cD = \{v_{\pm},v_1,v_2,v_3\}$ forms a semistable covering of $X^{\an}$ that is good with respect to the set of Weierstrass points.
Here are the dual graphs $\Gamma$ and $T$:
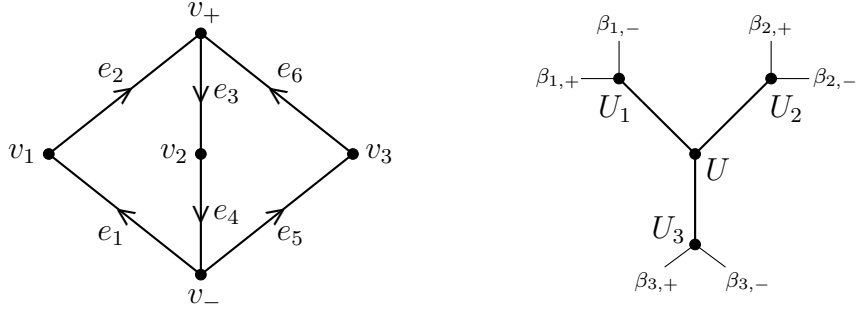
\[\begin{tikzpicture}
\draw [thick] (-4.5,1.6) -- (-6.5,0);
\draw [thick] (-4.5,1.6) -- (-4.5,0);
\draw [thick] (-4.5,1.6) -- (-2.5,0);
\filldraw (-4.5,1.6) circle (2pt);
\filldraw (-6.5,0) circle (2pt);
\filldraw (-4.5,0) circle (2pt);
\filldraw (-2.5,0) circle (2pt);
\filldraw (-4.5,-1.6) circle (2pt);
\draw [thick] (-4.5,-1.6) -- (-6.5,0);
\draw [thick] (-4.5,-1.6) -- (-4.5,0);
\draw [thick] (-4.5,-1.6) -- (-2.5,0);
\filldraw (-4.45,1.87) node{$v_+$};
\filldraw (-6.85,0) node{$v_1$};
\filldraw (-4.85,0) node{$v_2$};
\filldraw (-2.15,0) node{$v_3$};
\filldraw (-4.45,-1.93) node{$v_-$};
\node [thick,rotate=145] at (-5.5,-0.8) {>};
\node [thick,rotate=35] at (-3.5,-0.8) {>};
\node [thick,rotate=35] at (-5.5,0.8) {>};
\node [thick,rotate=145] at (-3.5,0.8) {>};
\node [thick,rotate=270] at (-4.5,0.8) {>};
\node [thick,rotate=270] at (-4.5,-0.8) {>};
\filldraw (-5.68,-1.08) node{$e_1$};
\filldraw (-5.68,1.08) node{$e_2$};
\filldraw (-4.16,0.8) node{$e_3$};
\filldraw (-4.16,-0.8) node{$e_4$};
\filldraw (-3.32,-1.08) node{$e_5$};
\filldraw (-3.32,1.08) node{$e_6$};

\qquad

\filldraw (2,0) circle (2pt);
\filldraw (3,1) circle (2pt);
\filldraw (1,1) circle (2pt);
\filldraw (2,-1.2) circle (2pt);
\draw [thick] (2,0) -- (3,1);
\draw [thick] (2,0) -- (1,1);
\draw [thick] (2,0) -- (2,-1.2);
\draw (2,-1.2) -- (1.6,-1.5);
\draw (2,-1.2) -- (2.4,-1.5);
\draw (1,1) -- (0.5,1);
\draw (1,1) -- (1,1.5);
\draw (3,1) -- (3.5,1);
\draw (3,1) -- (3,1.5);
\filldraw (2.3,-0.2) node{$U$};
\filldraw (0.95,0.6) node{$U_1$};
\filldraw (3.2,0.6) node{$U_2$};
\filldraw (1.65,-1) node{$U_3$};
\filldraw (0.2,1) node{\tiny{$\beta_{1,+}$}};
\filldraw (1,1.7) node{\tiny{$\beta_{1,-}$}};
\filldraw (3,1.7) node{\tiny{$\beta_{2,+}$}};
\filldraw (3.85,1) node{\tiny{$\beta_{2,-}$}};
\filldraw (1.5,-1.7) node{\tiny{$\beta_{3,+}$}};
\filldraw (2.7,-1.7) node{\tiny{$\beta_{3,-}$}};
\end{tikzpicture}\]
where $\beta_{j,\pm}$ denote the roots of $f_j$. Let $\tau\colon X^\an\to\Gamma$ be the retraction map.

Returning now to integrals, let $R=(1,2)$, $S=(1,-2)$ be points on $X$. We first compute integrals along a path from $S$ to $R$. Clearly, the points $R$ and $S$ belong, respectively, to the components $v_+$ and $v_-$. In order to pass from $v_-$ to $v_+$ (via $v_1$), we pick the following reference points:
\begin{align*}
P_{e_1} &= (a,4\cdot a + a^3 + 2\cdot a^5 + 4\cdot a^6 + a^7 + O(a^8)), \\
P_{e_2} &= (a,a + 4\cdot a^3 + 2\cdot a^5 + a^6 + 3\cdot a^7 + O(a^8)),
\end{align*}
where $a^4 = 5$. Now let $\gamma=\gamma^1\gamma^2\gamma^3$ where 
\begin{itemize}
    \item $\gamma^1$ is a path from $S$ to $P_{e_1}$ in $v_-$,
    \item $\gamma^2$ is a path from $P_{e_1}$ to $P_{e_2}$ in $v_1$, and
    \item $\gamma^3$ is a path from $P_{e_2}$ to $R$ in $v_+$,
\end{itemize}
so that $\tau(\gamma) = e_1e_2$. Our computations give
\[\BCint_\gamma \omega_i =  \begin{dcases}
\omit\hfil$2 \cdot a^{4} + 3 \cdot a^{8} + 4 \cdot a^{12} + 2 \cdot a^{16} + a^{20} + 2 \cdot a^{24} + O(a^{32})$\hfil &\text{if } i=0,\\
\omit\hfil$a^{4} + a^{8} + a^{12} + a^{24} + a^{28} + O(a^{32})$\hfil &\text{if } i=1,\\
\omit\hfil$a^{4} + 2 \cdot a^{24} + O(a^{32})$\hfil &\text{if } i=2,\\
\omit\hfil$1 + 3 \cdot a^{4} + 3 \cdot a^{8} + 2 \cdot a^{12} + 4 \cdot a^{16} + a^{20} + O(a^{32})$\hfil &\text{if } i=3,\\
\omit\hfil$3 + 4 \cdot a^{4} + 2 \cdot a^{8} + 4 \cdot a^{12} + 2 \cdot a^{16} + 2 \cdot a^{20} + a^{24} + 3 \cdot a^{28} + O(a^{32})$\hfil &\text{if } i=4.
\end{dcases}\]

Now, we compute the period integrals. The $1$-cycles
\[C_1 = e_1+e_2+e_3+e_4,\ \ \ C_2 = e_3+e_4+e_5+e_6\]
are a basis for $H_1(\Gamma;\C_p)$, and the tropical $1$-forms
\[\eta_1(e_i) =
\begin{dcases}
\omit\hfil$1/3$\hfil &\text{if } i=1\text{ or }2,\\
\omit\hfil$1/6$\hfil &\text{if } i=3\text{ or }4,\\
\omit\hfil$-1/6$\hfil &\text{if } i=5\text{ or }6;\\
\end{dcases}\ \ \ 
\eta_2(e_i) =
\begin{dcases}
\omit\hfil$-1/6$\hfil &\text{if } i=1\text{ or }2,\\
\omit\hfil$1/6$\hfil &\text{if } i=3\text{ or }4,\\
\omit\hfil$1/3$\hfil &\text{if } i=5\text{ or }6;\\
\end{dcases}\]
are a basis for $\Omega^1_{\trop}(\Gamma;\C_p)$ so that \[\tint_{C_i} \eta_j=
\begin{cases}
1 & \text{if }\ i=j,  \\
0 & \text{if }\ i\neq j.
\end{cases}\]
For each of the remaining edges, we pick the following reference points:
\begin{align*}
P_{e_3} &= (a+2,3\cdot a + a^2 + 2\cdot a^3 + a^4 + 4\cdot a^5 + O(a^7)), \\
P_{e_4} &= (a+2,2\cdot a + 4\cdot a^2 + 3\cdot a^3 + 4\cdot a^4 + a^6 + O(a^7)), \\
P_{e_5} &= (a+3,2\cdot a + a^2 + 4\cdot a^4 + 4\cdot a^5 + 3\cdot a^6 + O(a^7)), \\
P_{e_6} &= (a+3,3\cdot a + 4\cdot a^2 + a^4 + a^6 + a^7 + O(a^8)),
\end{align*}
recalling $a^4 = 5$. Let $\gamma_1=\gamma_1^1\gamma_1^2\gamma_1^3\gamma_1^4$ where
\begin{itemize}
    \item $\gamma_1^1$ is a path from $P_{e_1}$ to $P_{e_2}$ in $v_1$,
    \item $\gamma_1^2$ is a path from $P_{e_2}$ to $P_{e_3}$ in $v_+$,
    \item $\gamma_1^3$ is a path from $P_{e_3}$ to $P_{e_4}$ in $v_2$, and
    \item $\gamma_1^4$ is a path from $P_{e_4}$ to $P_{e_1}$ in $v_-$.
\end{itemize}
Then, by construction, $\gamma_1$ is a loop in $X^\an$ such that $\tau(\gamma_1)=C_1$ and the corresponding period integrals are
\[\BCint_{\gamma_1} \omega_i = \begin{dcases}
\omit\hfil$a^{8} + 3 \cdot a^{16} + a^{20} + O(a^{32})$\hfil &\text{if } i=0,\\
\omit\hfil$2 \cdot a^{4} + a^{12} + 3 \cdot a^{24} + 4 \cdot a^{28} + O(a^{32})$\hfil &\text{if } i=1,\\
\omit\hfil$a^{12} + 4 \cdot a^{16} + 3 \cdot a^{28} + O(a^{32})$\hfil &\text{if } i=2,\\
\omit\hfil$2 + 3 \cdot a^{4} + 2 \cdot a^{8} + 4 \cdot a^{16} + 2 \cdot a^{20} + a^{24} + a^{28} + O(a^{32})$\hfil &\text{if } i=3,\\
\omit\hfil$2 + 3 \cdot a^{4} + a^{8} + 2 \cdot a^{12} + 2 \cdot a^{16} + 4 \cdot a^{20} + 4 \cdot a^{24} + 4 \cdot a^{28} + O(a^{32})$\hfil &\text{if } i=4.
\end{dcases}\]

By constructing $\gamma_2$ analogously, we have
\[\BCint_{\gamma_2} \omega_i = \begin{dcases}
\omit\hfil$4 \cdot a^{4} + a^{8} + a^{12} + 2 \cdot a^{16} + 3 \cdot a^{20} + 3 \cdot a^{24} + 3 \cdot a^{28} + O(a^{32})$\hfil &\text{if } i=0,\\
\omit\hfil$a^{4} + 2 \cdot a^{8} + 3 \cdot a^{12} + 4 \cdot a^{16} + 4 \cdot a^{20} + 2 \cdot a^{24} + O(a^{32})$\hfil &\text{if } i=1,\\
\omit\hfil$2 \cdot a^{4} + 4 \cdot a^{8} + a^{12} + 3 \cdot a^{16} + a^{20} + 4 \cdot a^{24} + 4 \cdot a^{28} + O(a^{32})$\hfil &\text{if } i=2,\\
\omit\hfil$4 + a^{4} + 4 \cdot a^{8} + 2 \cdot a^{12} + 4 \cdot a^{16} + 4 \cdot a^{20} + a^{24} + 2 \cdot a^{28} + O(a^{32})$\hfil &\text{if } i=3,\\
\omit\hfil$3 \cdot a^{4} + 4 \cdot a^{8} + a^{16} + a^{20} + O(a^{32})$\hfil &\text{if } i=4.
\end{dcases}\]

Finally, using the comparison formula in Theorem~\ref{mainthm}, we compute
\begin{align*}
\Vint_S^R \omega_i &= \BCint_\gamma \omega_i - \left(\BCint_{\gamma_1} \omega_i\right)\left(\tint_{\tau(\gamma)} \eta_1\right) - \left(\BCint_{\gamma_2} \omega_i\right)\left(\tint_{\tau(\gamma)} \eta_2\right) \\
&= O(a^{32}) = O(5^{8}),\ \ \ i=0,1,2,4; \\
\Vint_S^R \omega_3 &= \BCint_\gamma \omega_3 - \left(\BCint_{\gamma_1} \omega_3\right)\left(\tint_{\tau(\gamma)} \eta_1\right) - \left(\BCint_{\gamma_2} \omega_3\right)\left(\tint_{\tau(\gamma)} \eta_2\right) \\
&= 1 + 3 \cdot a^{4} + a^{8} + 3 \cdot a^{12} + a^{16} + 3 \cdot a^{20} + a^{24} + 3 \cdot a^{28} + O(a^{32}) \\
&= 1 + 3 \cdot 5 + 5^2 + 3 \cdot 5^3 + 5^4 + 3 \cdot 5^5 + 5^6 + 3 \cdot 5^7 + O(5^{8})
\end{align*}
consistent with the fact that $\omega_0,\omega_1$ are regular but $\omega_3$ is not.\footnote{Incidentally, the integrals $\lVint_S^R \omega_2$ and $\lVint_S^R \omega_4$ vanish; the author does not know the reason behind this.}

Let $\gamma'$ (resp. $\gamma''$) be a path from $S$ to $R$ such that $\tau(\gamma') = (-e_4)(-e_3)$ (resp. $\tau(\gamma'') = e_5e_6$). As a consistency check, we replaced the path $\gamma$ in the computations above by $\gamma'$ and $\gamma''$, respectively; but these changes did not affect the (final) results, as expected.

\end{eg}

\bibliographystyle{amsalpha}
\bibliography{master}

\vspace{.2in}

\end{document}